\documentclass[11pt,a4paper,english]{article}

\makeatletter
\renewcommand\@biblabel[1]{}
\makeatother
\newcommand{\beq}{\begin{equation}}
\newcommand{\eeq}{\end{equation}}

\usepackage[latin1]{inputenc}
\usepackage{babel}
\usepackage[centertags]{amsmath}
\usepackage{amsfonts}
\usepackage{amssymb}
\usepackage{color}
\usepackage{amsthm}
\usepackage{newlfont}
\usepackage{graphicx}
\usepackage{dsfont}
\usepackage{authblk}

\DeclareMathOperator*{\essinf}{ess\,inf}

\newtheorem{theorem}{Theorem}[section]
\newtheorem{lemma}[theorem]{Lemma}

\newtheorem{proposition}[theorem]{Proposition}
\newtheorem{definition}[theorem]{Definition}
\newtheorem{remark}[theorem]{Remark}
\newtheorem{Assumptions}[theorem]{Assumption}

\makeatletter  
\@addtoreset{equation}{section}

\makeatother  
\begin{document}
\title{\textbf{Optimal Dynamic Procurement Policies \\ for a Storable Commodity \\ with L\'{e}vy Prices and Convex Holding Costs\thanks{Address correspondence to Gabriele Stabile, Dipartimento di Metodi e Modelli per l'Economia, il Territorio e la Finanza, Sapienza-Universit\`{a} di Roma, via del Castro Laurenziano 9, 00161 Roma, Italy. Email: gabriele.stabile@uniroma1.it.}}\footnote{Financial support granted by Sapienza-Universit\`{a} di Roma, Research Project \textquotedblleft Modelli, Valutazione e Gestione del Rischio per i Mercati delle Commodities\textquotedblright \ is gratefully acknowledged by the first and third authors. Financial support by the German Research Foundation (DFG) via grant Ri 1128-4-1 is gratefully acknowledged by the second author.}}\author[*]{M.B.~Chiarolla}
\author[**]{G.~Ferrari}
\author[*]{G.~Stabile}

\affil[*]{Dipartimento di Metodi e Modelli per l'Economia, il Territorio e la Finanza, Universit\`{a} di Roma 'La Sapienza', Roma, Italy}
\affil[**]{Center for Mathematical Economics, Bielefeld University, Bielefeld, Germany}
\date{}

\maketitle

\vspace{0.5cm}

{\textbf{Abstract.}} In this paper we study a continuous time stochastic inventory model for a commodity traded in the spot market and whose supply purchase  is affected by price
and demand uncertainty. A firm aims at meeting a random demand of the commodity at a random time by maximizing total expected profits. We model the firm's optimal procurement problem as a singular stochastic control problem in which controls are nondecreasing processes and represent the cumulative investment made by the firm in the spot market (a so-called stochastic `monotone follower problem'). We assume a general exponential L\'{e}vy process for the commodity's spot price, rather than the commonly used geometric Brownian motion, and general convex holding costs.

We obtain necessary and sufficient first order conditions for optimality and we provide the optimal procurement policy in terms of a \textsl{base inventory} process; that is, a minimal time-dependent desirable inventory level that the firm's manager must reach at any time.
In particular, in the case of linear holding costs and exponentially distributed demand, we are also able to obtain the explicit analytic form of the optimal policy and a probabilistic representation of the optimal revenue. The paper is completed by some computer drawings of the optimal inventory when spot prices are given by a geometric Brownian motion and by an exponential jump-diffusion process. In the first case we also make a numerical comparison between the value function and the revenue associated to the classical static ``newsvendor" strategy.

\smallskip


\smallskip

{\textbf{MSC2010 subsject classification}}: 93E20, 60G51, 49J40, 91B26.

\smallskip


\section{Introduction}
\label{introduction}

This paper contributes to the literature related to inventory management policies in the presence of price and demand uncertainty. Usually, the two main procurement mechanisms are the spot market, characterized by a negligible lead time, and long-term contracts, that offer the opportunity to procure the good for future use and with no payment prior to the delivery. In recent years the role of the spot market in procurement decisions has become prominent. That is also due to the raising of online spot markets such as memory chips, chemical, energy etc.\ (see Seifert, Thonemann, \& Hausman,\ 2004). As a consequence the pertinent literature has experienced an increasing interest in mathematical models for the spot market and its role in procurement decisions (see Bencherouf, 2007; Guo, Kaminsky, Tomecek, \& Yuen,\ 2011; Sato \& Sawaki, 2010, among others). For example, in Xiao-Li (2009) and Xinga, Wangb, and Liu (2012) the optimal procurement policies are obtained for a business-to-business (B2B) spot market. The optimal replenishment policy for an inventory model that minimizes the total expected discounted costs over an infinite planning horizon is studied in Bencherouf (2007) and Sato and Sawaki (2010) in the setting of impulsive control.
In Bencherouf (2007) the demand is driven by a Brownian motion with drift, whereas in Sato and Sawaki (2010) the demand is assumed to be deterministic and the market price of the good follows a geometric Brownian motion.

A model falling in the class of singular stochastic control problems (that is, a problem in which the controls are the cumulative amounts of the purchased/sold commodity and they are possibly singular with respect to the Lebesgue measure as functions of time) may be found in Guo et al.\ (2011). There the Authors study the inventory problem of a firm facing a random demand at the end of a random time interval. The firm may buy or sell the good in the spot market at any instant prior to the demand time. Trading the commodity in the spot market provides the immediate delivery of the good at a price modeled by a geometric Brownian motion. The firm aims at maximizing total net expected profits under linear holding and shortage costs. Holding costs are due to inventory storage, whereas shortage costs are incurred in when the demand exceeds the available stock. In fact, in such case the firm has a loss of revenue from not meeting the demand, and a loss of future business due to reputation lowering. On the other hand, there might be an excess of inventory at terminal time. That may be salvaged at a price possibly lower than the purchasing one through, for example, a discounted sale. It is shown that the spot market may be used to hedge against both supply costs and demand price uncertainty, although in the classical literature the firm usually enters the spot market when the demand exceeds the inventory level.

In this paper we essentially adopt the setting of Guo et al.\ (2011) but we take exponential L\'{e}vy prices and convex holding costs, and we assume that the commodity deteriorates over time at an exponential rate. However, we do not allow selling and we do not have inventory level's constraints.
Our control problem is a so-called `monotone follower problem' since selling is not allowed and the control process represents the cumulative amount of commodity purchased by the firm.
We model the commodity spot price $P$ by an exponential L\'{e}vy process to take into account the fact that, empirically, the market is non Gaussian but it exhibits significant skewness and kurtosis. We thus drop the assumption of normally distributed increments, while keeping a convenient Markovian structure. Our setup also includes the case of jump processes (like the Poisson process) or that of jump-diffusion processes, and of course the Brownian case typically assumed in the literature. As in Guo et al.\ (2011), we take a random demand $D$ fully described by a general absolutely continuous distribution function $F_D$ with finite mean, a random demand time $\Theta$ exponentially distributed, and we assume $D$ and $\Theta$ both independent of the price process. Such independence allows to rewrite the firm's problem in terms of an equivalent concave, monotone follower problem with infinite time horizon (some classical references on monotone follower problem are Karatzas, 1981; Karatzas, 1983; Karatzas, \& Shreve, 1984; El Karoui \& Karatzas, 1991).

We show existence and uniqueness of the optimal inventory policy by exploiting the concavity of our problem and by a suitable version of Koml\'os' Theorem (cf. Kabanov, 1999). Hence we characterize it by a set of necessary and sufficient first order conditions similar to those in Bank and Riedel (2001), Bank (2005), Chiarolla, Ferrari, and Riedel \ (2013), Ferrari (2015), Riedel and Su (2011) and Steg (2012), among others. Such conditions may be thought of as an infinite-dimensional, stochastic generalization of the classical Kuhn-Tucker conditions. They represent a valid alternative to the Hamilton-Jacobi-Bellman (HJB) approach, especially in non Markovian settings.

We show that the optimal policy requires to keep the inventory above a \textsl{base inventory} level $\ell^*_t$, at each time $t$. The process $\ell^*$ represents the maximal inventory level for which it is not profitable to delay the commodity's purchase to any future time, and hence it is characterized in terms of a family of optimal stopping problems (see also Bank \& F\"ollmer, 2002, for a relation with Gittins' indeces in continuous time). At times when the firm's inventory level is strictly above $\ell^*_t$, it is optimal to wait as at those times the firm faces excess of inventory. On the other hand, when the inventory is below $\ell^*_t$, then the firm should instantaneously invest in order to reach the level $\ell^*_t$. Hence, as expected in monotone follower problems, the optimal inventory policy behaves as the solution of a Skorohod's reflection problem at the (random) moving boundary $\ell^*$. Such a structure of the optimal policy is very natural from the point of view of inventory theory (see, e.g., Porteus, 1990) and it is well known in other contexts as well. See for example models with durable goods (cf., e.g., Bank \& Riedel, 2003), irreversible investments (cf.\ Chiarolla \& Haussmann, 2009; Riedel \& Su, 2011) and consumption with memory (cf.\ Bank \& Riedel, 2001). We point out that there is a specific condition on the parameters of our model under which the optimal procurement policy consists in not investing at all, whatever random demand $D$ and convex holding costs function one picks (see Proposition \ref{noinvestprop} below).

An explicit form of the optimal investment policy is obtained in Section \ref{examples} in the particular case of linear holding costs and exponentially distributed demand. The \textsl{base inventory} level is determined there by solving a backward stochastic equation in the spirit of Bank and El Karoui (2004) and, as expected, it turns out to be price-dependent. Moreover in such setting we provide a probabilistic representation of the value function. In the case of spot prices given by a geometric Brownian motion or by an exponential jump-diffusion process we make computer drawings of the optimal inventory level. Finally, we numerically compare the revenue associated to a modified version of the classical ``newsvendor" model, in which the commodity is purchased only once and at the initial time, with the value function of our model in the case of geometric Brownian motion prices. The results suggest that our optimal inventory strategy, dynamically acting over the entire given time interval, provides a higher revenue than that corresponding to the optimal static newsvendor one. Moreover, an increase in price volatility makes the difference among the two revenues increase. Therefore, although firms may prefer long term contracts when the commodity price is highly fluctuating, our results suggest that in order to increase the revenue the commodity must be dynamically procured in the spot market.

The rest of the paper is organized as follows. In Section \ref{firmproblem} we set the firm's optimal procurement problem, and in Section \ref{equivalent} we derive an equivalent concave singular stochastic control problem. In Section \ref{Existence} we prove the existence of optimal procurement policies and in Section \ref{FOCandbasecapacity} we obtain a characterization of them. Finally, Section \ref{examples} contains some explicit examples and computer drawings.


\section{Problem Formulation}
\label{firmproblem}

Consider a firm choosing a dynamic procurement policy of a single commodity to meet a random demand $D$ at a prescribed future random time $\Theta$. Each unit of demand satisfied gives rise to a profit depending on the spot price at time $\Theta$, $P_{\Theta}$. At any time $t \in [0,\Theta)$ the firm can instantaneously increase its inventory but it cannot buy inventory at terminal time $\Theta$ to meet demand. In case the inventory at time $\Theta$ is lower than the demand, then the firm incurs in a shortage cost proportional to $P_{\Theta}$.
On the other hand, when the inventory at time $\Theta$ exceeds the demand, the excess of inventory is salvaged by selling it in the spot market at price $\alpha_s P_{\Theta}$, for some $0 < \alpha_s \leq 1$. The costs associated to a procurement policy are the ordering cost given by the purchase of the commodity in the spot market, and the holding cost of storaging the commodity up to time $\Theta$.

To introduce the model fix a complete probability space $(\Omega,\mathcal{G},\mathbb{Q})$ and let the demand time $\Theta$ be a positive $\mathcal{G}$-measurable random variable. Consider a spot market in which the commodity may be traded at any time $t\in[0,\Theta)$ at a price $P_t$. Here $\{P_t, t\geq 0\}$ is an exogenous positive stochastic process on $(\Omega,\mathcal{G},\mathbb{Q})$ and $\{\mathcal{F}_t, t \geq 0\}$ is the filtration generated by $P_t$. We assume for $\mathbb{F}$ the usual hypotheses of completeness and right-continuity.
\begin{definition}
The price process $P_t$ has the following exponential L\'{e}vy structure
\beq
\label{P}
P_t=e^{\delta t -\zeta X_t - \pi(-\zeta)t}, \quad t \geq 0,
\eeq
where $\delta, \zeta \in \mathbb{R}$,
\label{assX}
$X$ is a Markov process with $X_0=0$, right-continuous sample paths, stationary, independent increments, and finite Laplace exponent $\pi(\cdot)$ given by
\[\mathbb{E}[e^{\zeta X_t}] = e^{\pi(\zeta)t}\] for all $\zeta \in \mathbb{R}$ and $t>0$.
\end{definition}
The constant $\delta$ may be seen as an interest rate, whereas $\zeta$ as a market price of risk.
\begin{remark}
\label{remarkP0}\hspace{10cm}
\begin{enumerate}
	\item There is no loss of generality in considering structure \eqref{P} since commonly used price dynamics may be cast in form \eqref{P} by adjusting $\delta$, through the Laplace transform of $X$. Such structure \eqref{P} usually arises in a financial market affected by a Markov uncertainty process $X$ (see, e.g., Duffie, 1992, Chapter 6, for a discussion in a Brownian setting).
	\item The uncertainty process $X$ is much more general than the Brownian one, commonly assumed in the literature. Indeed we drop the assumption of normally distributed increments, while keeping a convenient Markovian structure. Our setup also covers the case of jump processes, like the Poisson process, or that of jump-diffusion processes (see Section \ref{examples}), as well as the deterministic case obtained by setting $\zeta=0$.
\end{enumerate}
\end{remark}
\noindent For the positive, $\mathcal{G}$-measurable random variables $\Theta$ and $D$ we make the following (see also Guo et al., 2011, Section 3.1)
\begin{Assumptions}
\label{TD} \hspace{10cm}
\begin{enumerate}
	\item $\Theta$ is independent of the filtration $\{\mathcal{F}_t\}$ and $D$.
	\item $\Theta$ is exponentially distributed with rate $\lambda>0$.
	\item $D$ is independent of $\Theta$ and $\{\mathcal{F}_t\}$, and its distribution $F_D(y)=\mathbb{Q}(D \leq y)$ is absolutely continuous with density $f_D(y) = F'_D(y)$. Moreover
	$\mathbb{E}[D]< \infty$.
\end{enumerate}
\end{Assumptions}
\noindent In particular Assumption \ref{TD}, point 2 implies that the average demand time is $1/\lambda$.

We denote by $r>0$ the firm's manager discount factor and we assume for the firm's inventory $Y^{y,\nu}$ the following dynamics
\beq
\label{inventorydyn}
\left\{
\begin{array}{ll}
dY^{y,\nu}_t = -  \varepsilon Y^{y,\nu}_t dt + d\nu_t, \quad t>0,
\\
Y^{y,\nu}_0=y,
\end{array}
\right.
\eeq
where $\varepsilon\geq0$ is the commodity deterioration rate, $y\geq0$ is the initial inventory amount, and $\nu_t$ is the cumulative amount of commodity purchased up to time $t$. In fact, we take $\nu \in \mathcal{A}$ with
\begin{equation*}
\mathcal{A}:=\Big{\{}\nu:\Omega \times \mathbb{R}_{+} \rightarrow  \mathbb{R}_{+}\,\,\mbox{nondecreasing,\,\,left-continuous,}\,\,\{\mathcal{F}_t\}\mbox{-adapted} \ \mbox{s.t.} \ \nu_0=0\,\,\mathbb{Q}\mbox{-a.s.} \Big{\}}.
\end{equation*}
Then the explicit solution of \eqref{inventorydyn} is
\beq
\label{inventorydyn2}
Y^{y,\nu}_t=e^{-\varepsilon t}[ y + \overline{\nu}_t], \quad t \geq 0,
\eeq
with
\beq
\label{overlinenu}
\overline{\nu}_t:=\int_{[0,t)} e^{\varepsilon s}d\nu_s, \quad t \geq 0.
\eeq
The mapping $\mathcal{A} \ni \nu \mapsto \overline{\nu} \in \mathcal{A}$ defined by \eqref{overlinenu} is one-to-one and onto.  Left-continuity of $\nu$ captures the restriction that no commodity may be purchased at time $\Theta$. The fact that $\nu$ is $\{\mathcal{F}_t\}$-adapted guarantees that any investment decisions is taken on the basis of all the information on the price process up to time $t$.

The net gain function at the demand time $\Theta$ is given by $P_{\Theta}G(Y^{y,\nu}_{\Theta},D)$, where the `revenue multiplier' $G$  is defined by
\beq
\label{gainT}
G(y,D):=\alpha \min\{y,D\}- \alpha_p [D-y]^+ + \alpha_s[y-D]^+,
\eeq
for any inventory level $y \geq 0$.
Notice that $[D-Y^{y,\nu}_{\Theta}]^+$ represents the amount of unsatisfied demand at time $\Theta$ and $\alpha_p \geq 0$ is a penalty factor. Similarly $[Y^{y,\nu}_{\Theta}-D]^+$ is the excess amount of inventory at time $\Theta$ cleared in the spot market at a price possibly lower than $P_{\Theta}$, according to the factor $\alpha_s \in (0,1]$. Finally, $\alpha \geq 1$ represents a premium factor for the amount of demand satisfied. Notice that $\alpha + \alpha_p-\alpha_s \geq 0$.

The cost of increasing the inventory at time $t \in [0,\Theta)$ by a quantity $d\nu_t$ is $P_t \ d\nu_t$, whereas the cost of holding the inventory for an infinitesimal time interval $(t,t+dt) \subset [0,\Theta)$ is $c(Y^{y,\nu}_t)dt$. Hence, for any admissible procurement policy $\nu$, the total expected discounted return to the firm is
\begin{eqnarray*}
\lefteqn{\mathcal{J}(y,\nu) = \mbox{net\, expected\, discounted\, gain\, at\, demand\, time\, $\Theta$} } \\
&\hspace{1.5cm}-\,\mbox{total\, expected\, discounted\, holding\, costs}\,-\,\mbox{total\, expected\, ordering\, costs}; \nonumber
\end{eqnarray*}
that is,
\beq
\label{return}
\mathcal{J}(y,\nu) = \mathbb{E}\bigg[e^{-r \Theta}P_{\Theta} G(Y^{y,\nu}_{\Theta},D)-\int_{0}^{\Theta}e^{-r t}c(Y^{y,\nu}_t) dt -\int_{[0,\Theta)}e^{-r t}P_t d\nu_t\bigg].
\eeq
\noindent The following standing assumption shall hold throughout the paper
\begin{Assumptions}
\label{Assumptioncost}
The holding cost function $c : \mathbb{R}_{+} \rightarrow  \mathbb{R}_{+}$ is convex, strictly increasing, continuously differentiable with $c(0)=0$ and
\beq
\int_0^{\infty} e^{-rt}c'(ye^{-\varepsilon t})dt < \infty, \quad \text{for any } y\geq 0.
\label{assumption2.3}
\eeq
\end{Assumptions}

\begin{remark}
\label{Remarkasscost}
\hspace{10cm}
\begin{enumerate}
\item Notice that the requirement $c(0)=0$ is without loss of generality, since if $c(0)=K>0$ then one can always set $\hat{c}(y):=c(y)-c(0)$ and write $c(y)=\hat{c}(y)+K$, so that
the firms's optimization problem (cf.\ \eqref{valuefunction2} below) remains unchanged up to an additive constant.
	\item Convexity of $c(\cdot)$ together with $c(0)=0$ imply $c(y) \leq yc'(y)$ for any $y \geq 0$, and therefore
\beq
\label{integrabilityc}
\displaystyle \int_0^{\infty} e^{-rt}c(ye^{-\varepsilon t})dt \leq y\int_0^{\infty} e^{-(r+\varepsilon)t}c'(ye^{-\varepsilon t})dt < \infty,
\eeq
where finiteness of the last term in \eqref{integrabilityc} is due to \eqref{assumption2.3}.
   \item Natural examples of cost functions satisfying Assumption \ref{Assumptioncost} are $c(x)=\beta_1x$ or $c(x)=\beta_2 x^2$, for $\beta_1, \beta_2>0$.
   \item One could also allow proportional trading costs in the spot market, which amounts to a cost $K>0$ for each unit of commodity purchased, under suitable further conditions.
\end{enumerate}
\end{remark}
\noindent The firm aims at picking a procurement policy in order to maximize the total expected discounted return \eqref{return}.
For the well-posedness of the problem some integrability condition is needed.
In fact, we define the set $\mathcal{S}$ of admissible procurement policies
\begin{equation}
\label{admissibleset}
\mathcal{S}  := \Big{\{}\overline{\nu}\in \mathcal{A} \, \ : \ \,\mathbb{E}\Big[\int_{0}^{\infty}e^{- \beta t} P_t \ \overline{\nu}_t dt \Big]< \infty\Big{\}}
\end{equation}
with
\beq
\label{beta}
\beta:=r + \varepsilon + \lambda.
\eeq
Now in terms of $\overline{\nu}$ (cf.\ \eqref{overlinenu} and \eqref{inventorydyn2}) we may rewrite \eqref{return} as
\begin{eqnarray}
\label{return2}
\mathcal{J}(y,\nu) =\mathcal{\widetilde{J}}(y,\overline{\nu}) &\hspace{-0.25cm} =\hspace{-0.25cm}&  \mathbb{E}\bigg[e^{-r \Theta}P_{\Theta} G(e^{-\varepsilon \Theta}(y + \overline{\nu}_\Theta),D) \nonumber \\
& & \hspace{1cm} - \int_{[0,\Theta)}e^{-(r+\varepsilon) t}P_t d\overline{\nu}_t - \int_{0}^{\Theta}e^{-r t}c(e^{-\varepsilon t}(y + \overline{\nu}_t)) dt \bigg],
\end{eqnarray}
and by using the independence of $D$, $\Theta$ and $P$ (cf.\ Assumptions \ref{TD}), we have
\begin{eqnarray}
\label{return3}
\mathcal{\widetilde{J}}(y,\overline{\nu})  &\hspace{-0.25cm} =\hspace{-0.25cm} &  \mathbb{E}\bigg[\int_0^{\infty} \lambda e^{-(r+\lambda)t} P_t \bigg(\int_{0}^{\infty} G(e^{-\varepsilon t}(y + \overline{\nu}_t),z)f_D(z) dz\bigg)dt \nonumber \\
& & \hspace{0.6cm} - \int_{[0,\infty)}e^{-(r+\lambda + \varepsilon) t}P_t d\overline{\nu}_t - \int_{0}^{\infty}e^{-(r+\lambda) t}c(e^{-\varepsilon t}(y + \overline{\nu}_t)) dt\bigg].
\end{eqnarray}
Hence the firm's optimization problem is

\beq
\label{valuefunction2}
V(y)=\sup_{\overline{\nu} \in \mathcal{S}}\mathcal{\widetilde{J}}(y,\overline{\nu}), \qquad y\geq 0.
\eeq

From a mathematical point of view, problem \eqref{valuefunction2} falls into the class of singular stochastic control problems of monotone follower type (see the Introduction for some classical references), which allows controls possibly singular (as functions of time) with respect to the Lebesgue measure. If the performance criterion is concave (or convex) it is well known that the optimal control policy consists in keeping the state process at or above a certain threshold. In mathematical terms, the optimal control is the solution of a Skorohod reflection problem at a moving boundary (see, e.g., El Karoui \& Karatzas, 1991; Karatzas, 1981; Karatzas, 1983; Karatzas \& Shreve, 1984). Problem \eqref{valuefunction2} may be rewritten in terms of a new functional for which is easy to check concavity. This is accomplished in the following section (see Proposition \ref{properconcave} below).
\section{An Equivalent Concave Optimization Problem}
\label{equivalent}
Since the `revenue multiplier' $G$ has a singular point (i.e.\ a point in which it is not differentiable), we prefer to switch to a more regular function. In fact,
we rewrite the functional in \eqref{return3} by borrowing some ideas and arguments from Guo et al.\ (2011).\\
\indent Define the random field $\Gamma: \Omega \times \mathbb{R}_+ \times \mathbb{R}_+ \rightarrow \mathbb{R}$
\beq
\label{Pi}
\Gamma(\omega,t,y):=e^{-(r+\lambda)t}\Big[\lambda P_t(\omega) H(e^{-\varepsilon t} y)- c(e^{-\varepsilon t} y)\Big],
\eeq
with the function $H: \mathbb{R}_+ \rightarrow \mathbb{R}$ given by
\beq
\label{H}
H(y):=\alpha_s y - (\alpha_p-\alpha_s)\int_y^{\infty}zf_D(z)dz + (\alpha+\alpha_p-\alpha_s)y(1-F_D(y))+\alpha\int_0^y zf_D(z)dz,
\eeq
or equivalently
\beq
\label{H2}
H(y)=\alpha_s y + \alpha\mathbb{E}[D]- (\alpha+\alpha_p-\alpha_s)\int_y^{\infty}(z-y)f_D(z)dz.
\eeq
The following lemma shows that $\Gamma$ is continuously differentiable and concave in $y$.
\begin{lemma}
\label{propertiesPi}
The following properties hold,
\begin{description}
    \item[(i)] $H$ is continuously differentiable, strictly increasing, and concave in $\mathbb{R}_+$ with
\beq
\label{boundsH}
H(y) \leq \alpha_s y + \alpha\mathbb{E}[D], \qquad \qquad H(y) \geq -(\alpha_p-\alpha_s)\mathbb{E}[D];
\eeq
	\item[(ii)] $y \mapsto \Gamma(\omega,t,y)$ is continuously differentiable and concave for any $(\omega,t) \in \Omega \times \mathbb{R}_+$;
	\item[(iii)] $(\omega,t) \mapsto \Gamma(\omega,t,y)$ is $\{\mathcal{F}_t\}$-progressively measurable for any $y \geq 0$.
\end{description}
\end{lemma}
\begin{proof}
From \eqref{H} one has that
\beq
\label{derivativeH}
H'(y) = \alpha_s + (\alpha+\alpha_p-\alpha_s)(1-F_D(y)),
\eeq
which shows that $y \mapsto H(y)$ is continuously differentiable and strictly increasing as $\alpha+\alpha_p-\alpha_s\geq 0$.
Moreover
\beq
H''(y) = -(\alpha+\alpha_p-\alpha_s)f_D(y),
\eeq
which implies concavity of $H(\cdot)$.
The first of \eqref{boundsH} follows from \eqref{H2} recalling that $\alpha+\alpha_p-\alpha_s\geq 0$; whereas the second one follows from the fact that $H(\cdot)$ is increasing and $H(0)=-(\alpha_p-\alpha_s)\mathbb{E}[D]$.
As for ii) clearly $y \mapsto \Gamma(\omega,t,y)$ is continuously differentiable and concave for any $(\omega,t) \in \Omega \times \mathbb{R}_+$ since so are $H(\cdot)$ and $-c(\cdot)$ (cf.\ Assumption \ref{Assumptioncost}). Finally, to show iii) it suffices to notice that progressive measurability of $(\omega,t) \mapsto \Gamma(\omega,t,y)$, for any $y \geq 0$, is implied by the fact that $P$ is $\{\mathcal{F}_t\}$-progressively measurable being $\{\mathcal{F}_t\}$-adapted with right-continuous paths (cf.\ Definition \ref{assX}).
\end{proof}

As in Guo et al.\ (2011), Theorem 4, we obtain the decomposition of the optimal total expected discounted return $V(y)$ in terms of the value function of a new optimization problem and the expected value of the demand $\mathbb{E}[D]$. To obtain such decomposition we need the following
\begin{Assumptions}
\label{assr0}
$r + \lambda- \delta > 0.$
\end{Assumptions}
\noindent Notice that Assumption \ref{assr0} is equivalent to require that $\mathbb{E}[e^{-r \Theta}P_\Theta]<\infty$, i.e.\ the expected discounted price of a unit of commodity at the demand time is finite.
\begin{proposition}
Under Assumption \ref{assr0}
the optimal total expected return is
\beq
\label{equivalentvaluefunction0}
V(y) = W(y) - \frac{\lambda\alpha_s}{r+\lambda - \delta}\mathbb{E}[D],
\eeq
where
\beq
\label{W}
W(y):=\sup_{\overline{\nu} \in \mathcal{S}}\hat{\mathcal{J}}(y,\overline{\nu})
\eeq
and
\beq
\label{jhat}
\hat{\mathcal{J}}(y,\overline{\nu}) := \mathbb{E}\bigg[\int_{0}^{\infty} \Gamma(t, y + \overline{\nu}_t) dt
-\int_{[0,\infty)}e^{-\beta t}P_t d\overline{\nu}_t\bigg].
\eeq

\end{proposition}

\begin{proof}
Since
\begin{multline*}
\mathbb{E}\bigg[\int_0^{\infty} \lambda e^{-(r+\lambda)t} P_t \bigg(\int_{0}^{\infty} G(e^{-\varepsilon t}(y + \overline{\nu}_t),z)f_D(z) dz\bigg)dt\bigg] \\
\hspace{-18cm} =\mathbb{E}\bigg[\int_0^{\infty} e^{-(r+\lambda)t}\lambda P_t H(e^{-\varepsilon t}(y + \overline{\nu}_t))dt\bigg]-\frac{\lambda\alpha_s}{r+\lambda-\delta}\mathbb{E}[D]
\end{multline*}
by Assumption \ref{assr0}, then the proof follows from \eqref{Pi}.
\end{proof}
In order to show that the new value function $W$ of \eqref{W} is concave and proper we need a further assumption and some preliminary lemmas.
\begin{Assumptions}
\label{assr}
$\beta-\delta-\lambda\alpha_s > 0.$
\end{Assumptions}
\noindent Such assumption may be rewritten as
\[
\mathbb{E}\Big[ e^{-r \Theta} e^{-\varepsilon \Theta} \alpha_s P_{\Theta} \Big]=\frac{\lambda \alpha_s }{r+\varepsilon-\delta+\lambda}=\frac{\lambda \alpha_s }{\beta-\delta}<1
\]
and its economic interpretation is the following. At time $\Theta$ one unit of excess of inventory purchased at time $0$ amounts to $e^{-\varepsilon \Theta}$ (due to deterioration) and such amount is sold at the clearing price $\alpha_s P_{\Theta}$ (due to penalization), hence its expected discounted value at time $0$ is $\mathbb{E}\Big[ e^{-r \Theta} e^{-\varepsilon \Theta} \alpha_s P_{\Theta} \Big]$
and we require that such value is less than $1$, i.e. less than the unitary price at time $0$.\\
\indent Recall that $\beta=r + \varepsilon + \lambda$ (cf. \eqref{beta}) and notice that $\beta>\delta$.

\begin{lemma}
\label{lemma1}
Under Assumption \ref{assr0} the discounted price process $e^{-\beta t}P_t$ is a positive $\{\mathcal{F}_t\}$-supermartingale with right-continuous sample paths, such that
\[
\lim_{t \rightarrow \infty}e^{-\beta t}P_t = 0 \ \ \ \mathbb{Q}-a.s.
\]
\end{lemma}
\begin{proof}
The properties of the Markov process $X$ (cf.\ Definition \ref{assX}) together with $\beta>\delta$ imply the right-continuous supermartingale property of $e^{-\beta t}P_t$ since
\beq
\label{Psupermartingale}
\mathbb{E}[e^{-\beta t}P_t|\mathcal{F}_s] = e^{-(\beta-\delta)(t-s)}e^{-\beta s }P_s, \quad t \geq s \geq 0.
\eeq
Then Fatou's Lemma, \eqref{Psupermartingale} and $\beta>\delta$ imply
$$\displaystyle 0  \leq  \mathbb{E}[\lim_{t \rightarrow \infty}e^{-\beta t}P_t]  \leq  \displaystyle \liminf_{t \rightarrow \infty}\mathbb{E}[e^{-\beta t}P_t]=\lim_{t \rightarrow \infty}\mathbb{E}[e^{-\beta t}P_t]=0.$$
This together with
$\lim_{t \rightarrow \infty}e^{-\beta t}P_t\geq 0 \ \ \mathbb{Q}$-a.s. (cf.\ Karatzas \& Shreve, 1988, Ch.\ 1, Problem 3.16)
give $\lim_{t \rightarrow \infty}e^{-\beta t}P_t = 0 \ \ \mathbb{Q}$-a.s.
\end{proof}

From now on we will denote by $\tau$ any $\{\mathcal{F}_t\}$-stopping time with values in $[0,\infty]$ and, in light of Lemma \ref{lemma1}, we will adopt the following
\begin{definition}
$e^{-\beta \tau}P_{\tau}:= \lim_{t \rightarrow \infty}e^{-\beta t}P_t =0$ \,\,on \,\, $\{\tau=\infty\}.$
\end{definition}
\begin{lemma}

\label{lemma2}
Under Assumption \ref{assr0} one has
\begin{eqnarray}
 \label{Psupermartingale2}
 & & \mathbb{E}\bigg[\int_0^{\infty}e^{- \beta t} P_t dt\bigg] = \frac{1}{\beta -\delta}
 \end{eqnarray}
and, for any stopping time $\tau \in [0,\infty]$,
\begin{eqnarray}
 \label{Psupermartingale3}
 & & \mathbb{E}\bigg[\int_{\tau}^{\infty}e^{- \beta t} P_t dt\Big| \mathcal{F}_{\tau}\bigg] = e^{- \beta \tau} \frac{P_{\tau}}{\beta -\delta}.
 \end{eqnarray}
\end{lemma}
\begin{proof}
Applying Tonelli's Theorem and using the definition of $P$ (cf.\ Definition \eqref{P}) imply \eqref{Psupermartingale2}. As for \eqref{Psupermartingale3} notice that for any stopping time $\tau \in [0,\infty]$
 \begin{eqnarray*}
 \label{Psupermartingale3bis}
 & & \mathbb{E}\bigg[\int_{\tau}^{\infty}e^{- \beta  s} P_s ds\Big| \mathcal{F}_{\tau}\bigg] = e^{- \beta \tau} \int_0^{\infty}e^{-\beta t} \mathbb{E}[P_{t+\tau}\,| \mathcal{F}_{\tau}] dt \nonumber \\
 & & \hspace{1.5cm} = e^{-\beta \tau}P_{\tau} \int_0^{\infty}e^{-\beta t} \mathbb{E}\Big[\frac{P_{t+\tau}}{P_{\tau}}\,\Big| \mathcal{F}_{\tau}\Big] dt = e^{-\beta \tau}P_{\tau} \int_0^{\infty}e^{-\beta t} \mathbb{E}[P_t] dt \nonumber \\
 & & \hspace{1.5cm} = e^{-\beta \tau} \frac{P_{\tau}}{\beta-\delta},
 \end{eqnarray*}
where the third equality follows by the exponential form of the price process and the properties of the Markov process $X$ (cf.\ Definition \ref{assX}), whereas the last equality from \eqref{Psupermartingale2}.
\end{proof}

\begin{lemma}
\label{lemmaFubini}
Under Assumption \ref{assr0}, for any admissible procurement policy $\overline{\nu} \in \mathcal{S}$ it holds
\begin{equation}
\label{Fubini}
\mathbb{E}\bigg[\int_{0}^{\infty}e^{- \beta t} P_t \overline{\nu}_t dt \bigg] = \frac{1}{\beta-\delta}\,\mathbb{E}\bigg[\int_{[0,\infty)}^{}e^{- \beta t} P_t d\overline{\nu}_t\bigg].
\end{equation}
\end{lemma}
\begin{proof}
Fix $\overline{\nu} \in \mathcal{S}$, then by Tonelli's Theorem one has
\begin{eqnarray*}
\label{Tonelli1}
&&\int_{0}^{\infty}e^{- \beta t} P_t \overline{\nu}_t dt = \int_{0}^{\infty}e^{- \beta t} P_t \bigg(\int_{[0,t)} d\overline{\nu}_s \bigg) dt = \int_{[0,\infty)}^{} \bigg(\int_s^{\infty} e^{- \beta t} P_t dt \bigg) d\overline{\nu}_s.
\end{eqnarray*}
Taking expectations one obtains
\begin{eqnarray*}
\label{Tonelli2}
&&\mathbb{E}\bigg[\int_{0}^{\infty}e^{- \beta t} P_t \overline{\nu}_t dt\bigg] = \mathbb{E}\bigg[\int_{[0,\infty)}^{} \mathbb{E}\bigg[\int_s^{\infty} e^{- \beta t} P_t dt  \Big| \mathcal{F}_s\bigg]d\overline{\nu}_s\bigg] = \frac{1}{\beta-\delta}\mathbb{E}\bigg[\int_{[0,\infty)}^{}e^{- \beta t} P_t d\overline{\nu}_t\bigg],
\end{eqnarray*}
where the first equality follows from Dellacherie and Meyer (1982), Chapter VI, Theorem 57, whereas the last equality follows from \eqref{Psupermartingale3} with $\tau = s$.
\end{proof}

\begin{proposition}
\label{properconcave}
Under Assumptions \ref{assr0} and \ref{assr}, the new value function $W(\cdot)$ is concave and proper. Moreover
\begin{equation}
\label{boundonW}
W(y) \leq \frac{\lambda\alpha_s y}{\beta-\delta} + \frac{\lambda \alpha}{r+\lambda -\delta}\mathbb{E}[D].
\end{equation}
\end{proposition}
\begin{proof}
Recall that $\Gamma$ is concave by Lemma \ref{propertiesPi}. Then, using the affine property of $Y^{y,\overline{\nu}}$ in the control variable (cf.\ \eqref{inventorydyn2}), for any $\lambda \in (0,1)$, $y_1,y_2 \in \mathbb{R}_+$ and $\overline{\nu}_1, \overline{\nu}_2 \in \mathcal{S}$ we have
$$W(\lambda y_1 + (1-\lambda)y_2) \ \geq \ \hat{\mathcal{J}}(\lambda y_1 + (1-\lambda)y_2,\lambda \overline{\nu}_1 + (1-\lambda)\overline{\nu}_2) \ \geq \ \lambda \hat{\mathcal{J}}(y_1,\overline{\nu}_1) + (1 - \lambda)\hat{\mathcal{J}}(y_2,\overline{\nu}_2)$$
and concavity of $W(\cdot)$ follows.

We now show that $W$ is proper, that is $|W(y)| < \infty$ for any $y \geq 0$.
In fact
\begin{eqnarray}
\label{Wbiggerminusinfty}
W(y) \hspace{-0.25cm} & \geq & \hspace{-0.25cm} \hat{\mathcal{J}}(y,0) = \mathbb{E}\bigg[\int_{0}^{\infty} \lambda e^{-(r+\lambda)t}P_t\, H(ye^{-\varepsilon t})dt -\int_{0}^{\infty}e^{-(r+\lambda) t}c(ye^{-\varepsilon t}) dt\bigg]  \\
\hspace{-0.25cm} & \geq & \hspace{-0.25cm}  - \mathbb{E}\bigg[\lambda(\alpha_p-\alpha_s)\mathbb{E}[D]\int_{0}^{\infty}e^{-(r+\lambda)t}P_t dt + \int_{0}^{\infty}e^{-(r+\lambda) t}c(ye^{-\varepsilon t}) dt\bigg] > -\infty, \nonumber
\end{eqnarray}
where we have first used the second inequality in \eqref{boundsH}, then we have applied \eqref{integrabilityc}, Assumption \ref{TD} and Assumption \ref{assr0} to get the last inequality.

On the other hand, to show $W(y) < \infty$ for $y \geq 0$,
we fix an arbitrary $\overline{\nu} \in \mathcal{S}$ and we use the first inequality in \eqref{boundsH} and the fact that $c$ is positive to obtain
\begin{eqnarray}
\label{Wlessinfty}
& \hat{\mathcal{J}}(y,\overline{\nu}) &\leq \mathbb{E}\bigg[\int_{0}^{\infty} \lambda e^{-(r+\lambda)t}P_t \Big[\alpha_s e^{-\varepsilon t}(y + \overline{\nu}_t) + \alpha\mathbb{E}[D]\Big] dt -\int_{[0,\infty)}e^{-\beta t}P_t d\overline{\nu}_t\bigg] \nonumber \\
& &  =  \mathbb{E}\bigg[\int_{0}^{\infty} \lambda\alpha_s e^{- \beta t} P_t (y + \overline{\nu}_t) dt -\int_{[0,\infty)}e^{-\beta t}P_t d\overline{\nu}_t\bigg] + \frac{\lambda \alpha}{r+\lambda -\delta}\mathbb{E}[D]  \\
& &  = \frac{\lambda\alpha_s y}{\beta-\delta} - (\beta- \delta - \lambda\alpha_s)\mathbb{E}\bigg[\int_{0}^{\infty} e^{-\beta t} P_t \overline{\nu}_t dt\bigg] + \frac{\lambda \alpha}{r+\lambda -\delta}\mathbb{E}[D] \nonumber
\end{eqnarray}
where the last step follows from Lemma \ref{lemmaFubini}. Then passing to the supremum over $\overline{\nu} \in \mathcal{S}$ in \eqref{Wlessinfty} implies \eqref{boundonW} since $\beta- \delta - \lambda\alpha_s>0$  by Assumption \ref{assr}.
\end{proof}

\section{Existence of an Optimal Procurement Policy}
\label{Existence}
Existence of a solution $\overline{\nu}^*$ of concave (convex) singular stochastic control problems may be obtained by a suitable version of Koml\'os' Theorem (see for example Karatzas \& Wang, 2005; Riedel \& Su, 2011). In its classical formulation Koml\'os' Theorem (cf.\ Koml\'{o}s, 1967) states that if a sequence of random variables $\{Z_n, n \in \mathbb{N}\}$ is bounded from above in expectation, then there exists a subsequence $\{Z_{n_k}, k \in \mathbb{N}\}$ which converges a.e.\ in the Ces\`{a}ro sense to some random variable $Z$. Hence, if a maximizing (minimizing) sequence of admissible controls is Koml\'{o}s compact, then due to concavity (convexity) the limit provided by Koml\'os' Theorem turns out to be an optimal control policy. These arguments also work in our setting thanks to our assumptions.
Moreover, if $c$ is strictly convex, then $\hat{\mathcal{J}}(y,\cdot)$ of \eqref{jhat} is strictly concave in $\overline{\nu}$ for $y \geq 0$ fixed, and hence if a solution to \eqref{W} exists, then it is also unique.
\begin{theorem}
\label{existenceOC}
Let Assumptions \ref{assr0} and \ref{assr} hold. Then, for each fixed $y \geq 0$, there exists an optimal procurement policy $\overline{\nu}^{*}$ for problem \eqref{W}, i.e.\ $\hat{\mathcal{J}}(y,\overline{\nu}^{*})=W(y)$. Moreover, if $c(\cdot)$ is strictly convex then the optimal policy is unique (up to undistinguishability).
\end{theorem}

\begin{proof}
Take a maximizing sequence $\{\overline{\nu}^{(n)}\}_{n \in \mathbb{N}} \subset \mathcal{S}$; i.e.\ a sequence such that $\lim_{n \rightarrow \infty}\hat{\mathcal{J}}(y,\overline{\nu}^{(n)})=W(y)$.
Without loss of generality, we may take $\hat{\mathcal{J}}(y,\overline{\nu}^{(n)}) \geq W(y) - \frac{1}{n}$. Recall $\beta-\delta:=r+\lambda +\varepsilon - \delta > 0$ and use arguments similar to those in \eqref{Wlessinfty} to obtain, for any $n \in \mathbb{N}$,
\begin{eqnarray*}
\label{stima1}
\hspace{-0.6cm}W(y) - \frac{1}{n} & \leq &   \hat{\mathcal{J}}(y,\overline{\nu}^{(n)}) \; \leq \; \frac{\lambda \alpha_s y}{\beta-\delta} + \frac{\lambda \alpha}{r+\lambda -\delta}\mathbb{E}[D] - (\beta-\delta-\lambda\alpha_s)\mathbb{E}\bigg[\int_{0}^{\infty}e^{-\beta t} P_t \overline{\nu}^{(n)}_t dt\bigg],
\end{eqnarray*}
hence also
\beq
\label{bounded}
\hspace{-0.6cm}(\beta-\delta-\lambda\alpha_s)\mathbb{E}\bigg[\int_{0}^{\infty}e^{-\beta t} P_t \overline{\nu}^{(n)}_t dt\bigg] \ \leq \ \frac{\lambda \alpha_s y}{\beta-\delta}  + \frac{\lambda \alpha}{r+\lambda -\delta}\mathbb{E}[D]-W(y) + \frac{1}{n} \nonumber
\eeq
and
\beq
\label{L1}
\sup_{n \in \mathbb{N}}\mathbb{E}\bigg[\int_{0}^{\infty}e^{- \beta t} P_t \overline{\nu}^{(n)}_t dt\bigg] \  \leq \ \frac{1}{\beta-\delta-\lambda\alpha_s}\left(\frac{\lambda \alpha_s y}{\beta-\delta}  + \frac{\lambda \alpha}{r+\lambda -\delta}\mathbb{E}[D]-W(y) +1\right)
\eeq
where the right hand side is finite due to properness of $W$ (cf. Proposition \ref{properconcave}). Then by Lemma \ref{lemmaFubini} we obtain

\beq
\label{L1bis}
\sup_{n \in \mathbb{N}}\mathbb{E}\bigg[\int_{[0,\infty)}e^{- \beta t} P_t d\overline{\nu}^{(n)}_t\bigg] \leq \frac{\beta-\delta}{\beta-\delta-\lambda\alpha_s}\left(\frac{\lambda \alpha_s y}{\beta-\delta}  + \frac{\lambda \alpha}{r+\lambda -\delta}\mathbb{E}[D]-W(y) +1\right).
\eeq
Now we make a change of probability. In fact we define the equivalent probability measure $\widetilde{\mathbb{Q}}$ on $\mathcal{F}_t$ by setting
\beq
\label{Ptilde}
\frac{d\widetilde{\mathbb{Q}}}{d\mathbb{Q}}\Big|_{\mathcal{F}_t} = e^{-\delta t} P_t, \quad t \geq 0,
\eeq
and we denote by $\widetilde{\mathbb{E}}[\cdot]$ the expectation under $\widetilde{\mathbb{Q}}$. Then, in terms of the new probability measure, \eqref{L1bis} becomes
\beq
\label{L1tris}
\sup_{n \in \mathbb{N}}\widetilde{\mathbb{E}}\bigg[\int_{[0,\infty)}e^{- (\beta-\delta) t} d\overline{\nu}^{(n)}_t \bigg] \leq \frac{\beta-\delta}{\beta-\delta-\lambda\alpha_s}\left(\frac{\lambda \alpha_s y}{\beta-\delta}  + \frac{\lambda \alpha}{r+\lambda -\delta}\mathbb{E}[D]-W(y) +1\right).
\eeq
Then the sequence of nondecreasing, left-continuous, adapted processes $Z^{(n)}_t:=\int_{[0,t)} e^{- (\beta-\delta) s} d\overline{\nu}^{(n)}_s$ satisfies
\begin{equation}
\label{L1four}
\sup_{n \in \mathbb{N}}\widetilde{\mathbb{E}}\Big[Z^{(n)}_{\infty}\Big] < \infty.
\end{equation}
The mapping $\overline{\nu}^{(n)} \mapsto Z^{(n)}$ is one to one and onto, and its inverse gives $\overline{\nu}^{(n)}_t:=\int_{[0,t)} e^{(\beta-\delta) s} dZ^{(n)}_s$.

By a version of Koml\'{o}s Theorem for optional random measures\footnote{Let $\mathcal{V}_T$ denote the space of positive finite measures on $[0,T]$, $T\in (0,\infty]$, with the topology of weak-*convergence. Recall that an optional random measure is simply a $\mathcal{V}_T$-valued random variable $Z$ such that the process $Z_t(\omega):=Z(\omega;[0,t))$ is adapted.} (see Lemma 3.5 in Kabanov, 1999) there exists a subsequence $\{Z^{(n_k)}\}_{k \in \mathbb{N}} \subset \{Z^{(n)}\}_{n \in \mathbb{N}}$ and an optional random measure $Z^*$ such that $Z^{(n_k)}$ converges weakly in the Ces\`{a}ro sense to $Z^*$ a.s.; that is,
\beq
\label{Komlos}
\lim_{j \rightarrow \infty}\frac{1}{j}\sum_{k=1}^j \int_{[0,\infty)} f_s dZ^{(n_k)}_s = \int_{[0,\infty)}f_s dZ^*_s, \quad \widetilde{\mathbb{Q}}-a.s.
\eeq
for any bounded function $f:\mathbb{R}_+ \mapsto \mathbb{R}_+$ which is continuous $dZ^*$-a.e.\ in $\mathbb{R}_+$.
Then we may set $\overline{\nu}^*_t:=\int_{[0,t)} e^{(\beta-\delta) s}dZ^*_s=\int_{[0,\infty)} \mathds{1}_{[0,t)}(s)e^{(\beta-\delta) s} dZ^*_s$ and rewrite \eqref{Komlos} as
\beq
\label{Komlos2}
\lim_{j \rightarrow \infty} \xi^{(j)}_t = \overline{\nu}^*_t, \ \ \ \ \ \ \ \widetilde{\mathbb{Q}}-a.s.
\eeq
with $\xi^{(j)}_t:=\frac{1}{j}\sum_{k=1}^j \overline{\nu}^{(n_k)}_t$. Since each $\xi^{(j)}$ is a convex combination of the first $j$ elements of the subsequence of $\{\overline{\nu}^{(n)}\}_{n \in \mathbb{N}}$, we have
\beq
\label{newM}
\liminf_{j \rightarrow \infty} \widetilde{\mathbb{E}}\bigg[\int_{0}^{\infty}e^{-(\beta-\delta) t} \xi^{(j)}_t dt \bigg] < \infty
\eeq
by \eqref{L1}. Moreover, $t \mapsto \overline{\nu}^*_t$ is nondecreasing and hence the set of its points of discontinuity has zero Lebesgue measure; therefore Girsanov Theorem (cf.\ \eqref{Ptilde}), Fatou's lemma and \eqref{Komlos2} together with \eqref{newM} yield
$$\mathbb{E}\bigg[\int_{0}^{\infty}e^{- \beta t} P_t \overline{\nu}^*_t dt \bigg] = \widetilde{\mathbb{E}}\bigg[\int_{0}^{\infty}e^{- (\beta-\delta) t} \overline{\nu}^{*}_t dt\bigg]  < \infty.$$
It is possible to show that $\overline{\nu}^*$ admits a left-continuous modification which we still denote by $\overline{\nu}^*$.
We conclude that $\overline{\nu}$ is admissible, i.e.\ $\overline{\nu} \in \mathcal{S}$.\\
\indent In order to prove that $\overline{\nu}^*$ is optimal, it suffices to show that
\beq
\label{if}
\hat{\mathcal{J}}(y,\overline{\nu}^*) \geq W(y).
\eeq
We start from
\beq
\label{dom1}
\hat{\mathcal{J}}(y,\xi^{(j)}) = \mathbb{E}\bigg[\int_0^{\infty} \Big(\Gamma(t, y + \xi^{(j)}_t) -(\beta-\delta)e^{-\beta t} P_t \ \xi^{(j)}_t\Big) dt\bigg]
\eeq
where we have used Lemma \ref{lemmaFubini}. Then, under the new probability measure $\widetilde{\mathbb{Q}}$ (cf.\ \eqref{Ptilde}), we have
\beq
\label{dom2}
\hat{\mathcal{J}}(y,\xi^{(j)})= \widetilde{\mathbb{E}}\Big[\int_0^{\infty} \Phi(\xi^{(j)}_t) dt \Big]
\eeq
where $$\Phi(\xi^{(j)}_t):= \frac{e^{\delta t}}{P_t}\Gamma(t, y + \xi^{(j)}_t) - (\beta-\delta)e^{-(\beta-\delta) t}\xi^{(j)}_t.$$
Also, for each $j \in \mathbb{N}$,
$$\Phi(\xi^{(j)}_t) \leq \lambda \alpha_s y e^{-(\beta-\delta) t} + \lambda \alpha e^{-(r+\lambda -\delta)t}\mathbb{E}[D] \qquad \widetilde{\mathbb{Q}}-a.s.,$$
by \eqref{Pi}, the first inequality in \eqref{boundsH}, $c(\cdot) \geq 0$ (cf.\ Assumption \eqref{Assumptioncost}), and Assumption \ref{assr}. Therefore, by applying the reverse Fatou Lemma, \eqref{Komlos}, concavity of $\hat{\mathcal{J}}(y,\cdot)$ and Ces\`{a}ro Mean Theorem, we obtain
\beq
\label{dom3}
\hat{\mathcal{J}}(y,\overline{\nu}^*) \geq \limsup_{j \rightarrow \infty} \hat{\mathcal{J}}(y,\xi^{(j)}) \geq \limsup_{j \rightarrow \infty}\frac{1}{j}\sum_{k=1}^j \hat{\mathcal{J}}(y,\overline{\nu}^{(n_k)}) =W(y),
\eeq
hence $\overline{\nu}^*$ is optimal. As a subproduct, we also have that $\{\xi^{(j)}\}_{j \in \mathbb{N}}$ is itself a maximizing sequence of policies.

It remains to show that the optimal policy is unique (up to undistinguishability) if $c(\cdot)$ is strictly convex.
Let $\overline{\nu}^{*,1}$ and $\overline{\nu}^{*,2}$ be two optimal strategies and define the admissible procurement strategy $\hat{\nu}:=\frac{1}{2}\overline{\nu}^{*,1} + \frac{1}{2}\overline{\nu}^{*,2}$. We then have
\begin{eqnarray*}
0 & \hspace{-0.25cm}\geq \hspace{-0.25cm}& \hat{\mathcal{J}}(y,\hat{\nu}) - W(y) = \hat{\mathcal{J}}(y,\hat{\nu}) - \frac{1}{2}\hat{\mathcal{J}}(y,\overline{\nu}^{*,1}) - \frac{1}{2} \hat{\mathcal{J}}(y,\overline{\nu}^{*,2}) \nonumber \\
& \hspace{-0.25cm} = \hspace{-0.25cm}& {\mathbb{E}}\bigg[\int_0^{\infty} \Big(\Gamma(t, y + \hat{\nu}_t) - \frac{1}{2}\Gamma(t, y + \overline{\nu}^{*,1}_t) - \frac{1}{2}\Gamma(t, y + \overline{\nu}^{*,2}_t)\Big) dt\bigg] \geq 0, \nonumber
\end{eqnarray*}
where the last inequality follows by concavity of $\Gamma(t, \cdot)$ (cf. Lemma \eqref{propertiesPi}).
Thus the inequalities above must be equalities and again by concavity of $\Gamma(t, \cdot)$ it must be
$$
\Gamma(t, y + \hat{\nu}_t)=\frac{1}{2}\Gamma(t, y + \overline{\nu}^{*,1}_t) +\frac{1}{2}\Gamma(t, y + \overline{\nu}^{*,2}_t), \quad \mathbb{Q}-\text{a.s.\,\, for \,\,a.e.}\ t\geq 0.
$$
Hence $\overline{\nu}^{*,1}_t = \overline{\nu}^{*,2}_t$, $\mathbb{Q}$-a.s. for a.e.\ $t\geq0$, by noticing that
the assumption of strictly convexity of $c(\cdot)$ implies the strict concavity of $\Gamma(t, \cdot)$.
Therefore, by left-continuity of $\overline{\nu}^{*,i}$, $i=1,2$, we conclude that $\overline{\nu}^{*,1}$ and $\overline{\nu}^{*,2}$ are indistinguishable.
\end{proof}

\begin{remark}
\label{Remaexistence}
Notice that the existence of an optimal procurement strategy may be obtained for random demand times $\Theta$ with continuous density more general than the exponential one as long as their hazard rate $h_t$ satisfies the condition
\[
h_t<\frac{r+\varepsilon -\delta}{\alpha_s},
\]
needed to prove properness of the value function $W(y)$.
Such condition is a quite strong requirement on the hazard rate $h_t$, which is often unbounded, although it is satisfied by
some well known distributions as the log-normal or the log-logistic distributions for suitable parameters.
\end{remark}


\section{Characterization of Optimal Procurement Policies}
\label{FOCandbasecapacity}

Theorem \ref{existenceOC} provides the existence of an optimal policy $\overline{\nu}^* \in \mathcal{S}$ for the procurement problem without giving any information about its nature.
In this section we provide a complete characterization of the optimal procurement policy through
generalized stochastic first order conditions. Such approach does not need \textsl{a priori} smoothness of the value function, neither of the boundary between the investment and the no-investment regions, in order to determine the optimal policy. Therefore it allows to overcome the regularity issues dealt with in the HJB quasi-variational inequality approach arising from our two-dimensional setting with state process $(Y_t,P_t)$.

Characterizing the optimal controls by first order conditions appeared in Bertola (1998) for a profit maximizing firm with Cobb-Douglas operating profit function and uncertainty given by a geometric Brownian motion. Such concept
was then developed in Bank \& Riedel (2001) in the case of a static budget constraint, in Bank (2005) in the case of a stochastic dynamic finite-fuel constraint, in Steg (2012) in a capital accumulation game of sequential irreversible investment, and in Chiarolla et al.\ (2013) for an optimization problem involving N firms.
In Riedel and Su (2011) a stochastic first order conditions approach was employed to obtain the optimal policy through the solution to a backward stochastic equation.

For any admissible $\nu$ the super-gradient process $\nabla \hat{\mathcal{J}}(y,\nu)$ is defined as the unique optional\footnote{A stochastic process is optional if it is measurable with respect to the optional sigma-algebra $\mathcal{O}$ on $\Omega \times [0,T]$. The optional sigma-algebra $\mathcal{O}$ is generated by, e.g., the right-continuous and adapted processes (see, e.g., Dellacherie \& Meyer, 1982).} process satisfying
\beq
\label{superg}
\nabla \hat{\mathcal{J}}(y,\overline{\nu})(\tau):=\mathbb{E}\bigg[ \int_{\tau}^{\infty} \Gamma_y(t,y + \overline{\nu}_t)dt \Big|\mathcal{F}_{\tau}\bigg] - e^{-\beta \tau}P_{\tau},
\eeq
for any $\{\mathcal{F}_t\}$-stopping time $\tau \in [0,\infty]$, with $\beta:=r+\lambda + \varepsilon$ (cf.\ \eqref{beta}) and $\Gamma_y:=\partial \Gamma/\partial y$.
The super-gradient \eqref{superg} may be interpreted as the marginal expected net profit resulting from the purchase of an extra unit of commodity at time $\tau$.
\begin{theorem}
\label{thfoc}
Let Assumptions \ref{assr0} and \ref{assr} hold and fix $y\geq 0$. Then, for an admissible policy $\overline{\nu}^{*}$, the following first order conditions
\beq
\label{FOCs}
\left\{
\begin{array}{ll}
\displaystyle \nabla \hat{\mathcal{J}}(y,\overline{\nu}^*)(\tau)  \leq 0 \quad \ \ \ \ \ \ \ \ \ \ \ \textrm{a.s.\ for all stopping times} \ \tau \in [0,\infty], \\ \\
\displaystyle \int_{0}^{\infty} \nabla \hat{\mathcal{J}}(y,\overline{\nu}^*)(t) \ d\overline{\nu}^{*}_t = 0 \quad \ \textrm{a.s.}
\end{array}
\right.
\eeq
are necessary and sufficient for optimality of $\overline{\nu}^{*}$ in problem \eqref{W}.
\end{theorem}
\begin{proof}
Set $k_t:=e^{-\beta t}P_{t}$ and apply Steg (2012), Proposition $3.2$ which requires the three conditions of his Assumption 3.1. Our
Lemma \ref{propertiesPi} and Proposition \ref{properconcave} guarantee the first two; as for the third one, it asks for $y \mapsto \Gamma_y(\omega,t,y)$ strictly decreasing
but it is easily checked that it suffices concavity of $\Gamma(\omega,t,\cdot)$ to prove optimality of the first order conditions.
\end{proof}

The intuition of the first order conditions is that when the gradient is positive at some stopping time, a small extra investment is profitable.
\begin{remark}
\label{heuristic}
The first order conditions \eqref{FOCs} can be thought of as a stochastic infinite-dimensional generalization of the classical Kuhn-Tucker conditions.
In real analysis, if one deals with an optimization problem of the form $\max_{x \geq 0}f(x)$,
for some smooth concave real function $f$, then it is well known that the Kuhn-Tucker conditions for optimality of $x^*$ are
\beq
\label{FOCrealanalysis}
f'(x^*) \leq 0, \qquad \qquad x^*f'(x^*) =0.
\eeq
Roughly speaking, in our setting the nonnegativity constraint $x \geq 0$ is replaced by the irreversibility constraint $d\overline{\nu}_t\geq 0$ for all $t \geq 0$, $\mathbb{Q}-$a.s., and the role of the first derivative $f'(\cdot)$ is played by the super-gradient process $\nabla \hat{\mathcal{J}}(y,\overline{\nu})$. Then \eqref{FOCs} is a generalization of \eqref{FOCrealanalysis}.
\end{remark}

The following proposition shows with the help of Theorem \ref{thfoc} that it is optimal not to exercise the investment option at all if $\beta-\delta$ is sufficiently large.

\begin{proposition}
\label{noinvestprop}
Under Assumptions \ref{assr0} and \ref{assr}, if $\beta-\delta \geq \lambda(\alpha_p + \alpha)$ then, for any $y\geq 0$, it is never optimal to invest in the commodity, i.e.\ $\overline{\nu}^*\equiv0$ for any $t\geq 0$ $\mathbb{Q}-$a.s. Hence
$$W(y)=\mathbb{E}\bigg[\int_{0}^{\infty} \Gamma(t, y) dt\bigg]= \int_{0}^{\infty} e^{-(r+\lambda) t}\Big[e^{\delta t} \lambda H(e^{-\varepsilon t}y) - c(e^{-\varepsilon t}y)\Big]dt.$$
\end{proposition}
\begin{proof}
For any admissible procurement policy $\overline{\nu}$ and any stopping time $\tau \geq 0$ we have
\begin{eqnarray}
\label{nablazeroinvest}
& & \displaystyle \nabla \hat{\mathcal{J}}(y,\overline{\nu})=\mathbb{E}\bigg[ \int_{\tau}^{\infty} \Gamma_y(t,y + \overline{\nu}_t)dt \Big|\mathcal{F}_{\tau}\bigg] - e^{-\beta \tau}P_{\tau} \nonumber \\
& & = \lambda (\alpha_p + \alpha) \mathbb{E}\bigg[ \int_{\tau}^{\infty} e^{-\beta t} P_t dt \Big|\mathcal{F}_{\tau}\bigg] - \lambda(\alpha +\alpha_p-\alpha_s)\mathbb{E}\bigg[ \int_{\tau}^{\infty} e^{-\beta t}P_t F_D(e^{-\varepsilon t}(y + \overline{\nu}_t)) dt\Big|\mathcal{F}_{\tau}\bigg] \nonumber \\
& & \ \ \ - \mathbb{E}\bigg[ \int_{\tau}^{\infty} e^{-\beta t} c'(e^{-\varepsilon t}(y + \overline{\nu}_t)) dt\Big|\mathcal{F}_{\tau}\bigg] - e^{-\beta \tau}P_{\tau} \\
& & < \lambda (\alpha_p + \alpha)\mathbb{E}\bigg[ \int_{\tau}^{\infty} e^{-\beta t} P_t dt \Big|\mathcal{F}_{\tau}\bigg] - e^{-\beta \tau}P_{\tau} \ =  \ e^{-\beta \tau}P_{\tau}\Big(\frac{\lambda (\alpha_p + \alpha)- (\beta-\delta)}{\beta-\delta}\Big), \nonumber
\end{eqnarray}
with $\beta:=r+\varepsilon + \lambda$, where the inequality in the third step follows from the nonnegativity of $\alpha +\alpha_p-\alpha_s$, of $F_D(\cdot)$ and of $c'(\cdot)$ (cf.\ Assumption \ref{Assumptioncost}), whereas the last equality follows from \eqref{Psupermartingale3}. It follows that if $\lambda (\alpha_p + \alpha) \leq \beta-\delta$, then the super-gradient is always strictly negative for any admissible procurement policy $\overline{\nu}$ and any stopping time $\tau \geq 0$, and hence it is never optimal to purchase commodity (cf. Theorem \ref{thfoc}).
\end{proof}
Theorem \ref{existenceOC} guarantees the existence of an optimal investment policy for any $y \geq 0$, and
the first order conditions \eqref{FOCs} completely characterize it. However, such conditions are not necessarily binding and hence they do not always explicitly determine the optimal policies.
For that, we follow Riedel and Su (2011), Section 3, or Steg (2012), Proposition 3.3, and we obtain the optimal procurement policies in terms of a process\footnote{Notice that in Riedel and Su (2011), Section 3, or Steg (2012), Section 3.2, such process is referred to as \textsl{base capacity} process.} $l^*$, which we refer to as \textsl{base inventory} process, representing a desirable value of inventory that the firm aims to reach at every time.
Recall that $\beta:=r+\lambda +\varepsilon$.
\begin{theorem}
\label{bsecapacitythm}
Assume that Assumptions \ref{assr0} and \ref{assr} hold. Then there exists an optional process $\ell^{*}$ taking values in $[0,\infty)$ and satisfying, for $t \geq 0$,
\beq
\label{basecapacitydef}
\ell^{*}_t=\sup\Big\{ z \in \mathbb{R}_{+} \Big | \essinf_{\tau \geq t }\mathbb{E}\Big[\int_t^{\tau}\Gamma_y(s, ze^{-\epsilon s}) ds + e^{-\beta \tau}P_{\tau} \Big| \mathcal{F}_{t}\Big] = e^{-\beta t}P_{t}\Big\} \vee 0, \  \ \mathbb{Q}\textrm{-a.s.,}
\eeq
such that, for any $y \geq 0$, the procurement policy
\beq
\label{optimalprocurement}
\overline{\nu}^{*}_t := \sup_{0 \leq u < t}\Big(e^{\varepsilon u}\ell^*_u - y\Big) \vee 0, \qquad t \geq 0,
\eeq
is optimal for \eqref{W}, if it is admissible.
\end{theorem}
\begin{proof}
Since Theorem \ref{existenceOC} provides the existence of an optimal procurement policy ${\nu}^{*}$ for any $y \geq 0$, it suffices to apply Steg (2012), Proposition 3.3 (which may be easily adapted to our case with depreciation rate $\varepsilon \geq 0$).
\end{proof}
\noindent The base inventory process $\ell^*$ is the maximal inventory level for which it is not profitable to delay marginal purchase of the commodity to any future stopping time.
Also, for any $t\geq 0$, $\ell^*_t$ is uniquely defined if $y\mapsto\Gamma_y(t,y)$ is strictly decreasing, i.e.\ if the holding cost function $c(\cdot)$ is strictly convex (cf.\ \eqref{Pi}).
The optimal procurement policy provided by \eqref{optimalprocurement} consists in keeping the inventory level $Y^{y,\overline{\nu}^*}$ always at or above $\ell^*_t$. If the inventory level at time $t$ is such that $Y^{y,\overline{\nu}^*}_t > \ell^*_t$, then the firm faces excess inventory and should wait to buy more commodity. If the inventory level is below $\ell^*_t$, then the firm should invest $\overline{\nu}^*_t = \ell^*_t - Y^{y,\overline{\nu}^*}_t$ in order to reach the level $\ell^*_t$. Such property of the optimal policy is quite natural in inventory theory (see, e.g., Porteus, 1990).

The signal process $\ell^*_t$ may be characterized as the unique optional positive solution of a backward stochastic equation related to \eqref{FOCs} (cf.\ the Bank-El Karoui's representation problem in Bank \& El Karoui, 2004, Theorem 1 and Theorem 3). In the context of stochastic irreversible capacity expansion problems, such backward equation has been obtained in Riedel and Su (2011), Chiarolla and Ferrari (2014), Ferrari (2015) (here the backward equation is actually an integral equation) under the Inada conditions on the corresponding $\Gamma$  (cf.\ Inada, 1963), i.e.
$$\lim_{y\downarrow 0}\Gamma_y(t,y)=\infty \quad \text{and} \quad \lim_{y\uparrow \infty}\Gamma_y(t,y) =0 \ \ \ \ \ \ \ \textrm{for} \ t\geq 0.$$
In the present setting, although our $\Gamma$ (cf.\ \eqref{Pi}) does not satisfy the Inada conditions, we adapt some of their arguments to characterize $\ell^*_t$ through a backward stochastic equation, which we actually manage to solve in the case of linear holdings costs and exponentially distributed demand (cf. Section \ref{OptimalProcPolicy} below).

\begin{proposition}
\label{baseandoptimalcontrol}
If there exists a progressively measurable process $\ell^*$ solving the backward stochastic equation
\beq
\label{backwardgeneric}
\mathbb{E}\bigg[\int_{\tau}^{\infty} \Gamma_y(t,\sup_{\tau \leq u < t} (e^{\varepsilon u}\ell^*_u)) dt \Big|\mathcal{F}_{\tau}\bigg] = e^{-\beta \tau}P_{\tau}, \ \ \ \ \textrm{a.s.\ for any stopping time} \ \tau \in [0,\infty],
\eeq
then the procurement policy
\beq
\label{optimalcontrolgeneric}
\overline{\nu}^*_t:= \sup_{0 \leq u < t}(e^{\varepsilon u}\ell^*_u -y) \vee 0, \qquad  \qquad \overline{\nu}^*_0=0,
\eeq
is optimal for problem \eqref{W}, if it is admissible.
\end{proposition}
\begin{proof}
We borrow arguments from Bank and Riedel (2001) (or Riedel \& Su, 2011). It suffices to show that the process
\beq
\label{optimalinventorygeneric}
Y^*_t:= Y^{y,\overline{\nu}^*}_t = e^{-\varepsilon t}[y + \overline{\nu}^*_t] = ye^{-\varepsilon t} \vee \sup_{0 \leq u < t}(e^{-\varepsilon(t-u)}\ell^*_u)
\eeq
satisfies the first-order conditions for optimality \eqref{FOCs}.
For any stopping time $\tau$, \eqref{superg} together with concavity of $\Gamma(t,\cdot)$ imply
\begin{eqnarray}
\label{checkoptimality}
\nabla \hat{\mathcal{J}}(y,\overline{\nu}^*)(\tau) \hspace{-0.25cm} & = & \hspace{-0.25cm} \mathbb{E}\bigg[\int_{\tau}^{\infty} \Gamma_y(t, y \vee \sup_{0 \leq u < t}(e^{\varepsilon u}\ell^*_u)) dt \Big | \ \mathcal{F}_{\tau}\bigg] -e^{-\beta \tau}P_{\tau}  \\
\hspace{-0.25cm} & \leq & \hspace{-0.25cm} \mathbb{E}\bigg[\int_{\tau}^{\infty} \Gamma_y(t, \sup_{\tau \leq u < t}(e^{\varepsilon u}\ell^*_u))\Big| \mathcal{F}_{\tau}\bigg]- e^{-\beta \tau}P_{\tau} = 0, \nonumber
\end{eqnarray}
since $\ell^*$ solves \eqref{backwardgeneric}. \\
\indent To prove the second condition in \eqref{FOCs}, notice that for any $\tau \geq 0$ at which the purchase of the commodity takes place (i.e.\ where $d\overline{\nu}^*_{\tau}:=\overline{\nu}^*_{\tau+\varepsilon} - \overline{\nu}^*_{\tau} >0$, for every $\varepsilon > 0$) we have
$$
Y^*_t = e^{-\varepsilon t} \sup_{\tau \leq u < t}(e^{\varepsilon u}\ell^*_u) \ \ \ \ \ \textrm{for} \ \ t > \tau,
$$
by \eqref{optimalinventorygeneric}. Therefore $y + \overline{\nu}^*_t=\sup_{\tau \leq u < t}(e^{\varepsilon u}\ell^*_u)$ for $t>\tau$, hence at times $\tau$ carrying the random Borel measure $d\nu^*$ we have $\nabla \hat{\mathcal{J}}(y,\overline{\nu}^*)(\tau)=0$ by \eqref{backwardgeneric}.
\end{proof}
\begin{remark}
Clearly $\overline{\nu}^*$ of \eqref{optimalcontrolgeneric} is nondecreasing and left-continuous; it is $\{\mathcal{F}_t\}$-progressively measurable (since $l^*$ is so, cf.\ Dellacherie $\&$ Meyer, 1982, Theorem IV.33), hence
also $\{\mathcal{F}_t\}$-adapted. Therefore $\overline{\nu}^*$ is admissible if and only if the integrability condition $\mathbb{E}\big[\int_{0}^{\infty}e^{- \beta t} P_t \overline{\nu}^*_t dt \big]< \infty$ holds. Such condition must be checked on a case by case basis but it is usually satisfied if $\beta$ is sufficiently large.
\end{remark}
Notice that, as discussed in Section 3 of Riedel and Su (2011), \eqref{backwardgeneric} shows clear similarities with the first order conditions \eqref{FOCs}. In fact, by considering the first order conditions for a firm that starts investing at time $\tau$, we may take the supremum from time $\tau$ on in the inventory that tracks the base inventory $\ell^*$ (cf.\ \eqref{optimalprocurement}) and then plug it into the supergradient \eqref{superg}.
Then the first order conditions are binding for the firm and thus we obtain an equality as in \eqref{backwardgeneric}.

In applications Proposition \ref{baseandoptimalcontrol} turns out to be very useful as it provides a constructive method to find $\ell^*$  and hence solve problem \eqref{W}. In fact, \eqref{backwardgeneric} may be at least found numerically by backward induction on a discretized version of problem \eqref{backwardgeneric} (see Bank \& F\"ollmer, 2002, Section 4). In the next section we explicitly solve the backward stochastic equation \eqref{backwardgeneric} in the case of an exponentially distributed demand (cf.\ Lariviere \& Porteus, 1990; Zhang, Nagarajan, \& So\v{s}i\'c, 2009) for the use of the exponential distribution in inventory management literature) and linear holding costs (cf.\ Guo et al.,\ 2011; Tarima \& Kingsman, 2004; Zhang, 2010, among others).


\section{Explicit Results: Linear Holding Costs}
\label{examples}

Throughout this section Assumptions \ref{TD}, \ref{Assumptioncost}, \ref{assr0} and \ref{assr} still hold true. We also assume, as in Guo et al.\ (2011), linear holding costs, i.e. $c(x)=c x$ for some $c>0$, and zero deterioration rate, i.e. $\varepsilon=0$. Moreover, we assume that $D$ is exponentially distributed with parameter $\gamma>0$; that is, $f_D(z) = \gamma e^{-\gamma z}$.
Within this setting one has $\beta=r+\lambda$ and (cf.\ \eqref{H})
\beq
\label{Hlin}
H(y) = \alpha_s y + \frac{\alpha}{\gamma} - \frac{(\alpha_p + \alpha - \alpha_s)}{\gamma}e^{-\gamma y},
\eeq
and (cf.\ \eqref{Pi})
\beq
\label{Pilin}
\Gamma(t,y)=e^{-\beta t}\lambda P_t\Big[\alpha_s y + \frac{\alpha}{\gamma} - \frac{(\alpha_p + \alpha - \alpha_s)}{\gamma}e^{-\gamma y}\Big] - e^{-\beta t} c y .
\eeq
Then,
\beq
\label{Piderivative}
\Gamma_y(t,y) = e^{-\beta t}\Big[\lambda P_t\Big(\alpha_s + (\alpha + \alpha_p -\alpha_s)e^{-\gamma y}\Big) - c\Big].
\eeq

\subsection{The Optimal Procurement Policy}
\label{OptimalProcPolicy}

We now find, in our general exponential L\'evy setting, the explicit form of the optimal procurement policy, which turns out to be bounded. To the best of our knowledge, such explicit result appears here for the first time.

\begin{proposition}
\label{backlinearcostsprop}
With $\beta=r + \lambda$, let $\tau_{\beta}$ be an exponentially distributed random time, independent of $P$, with parameter $\beta$, and set
\beq
\label{kappa}
\kappa:=\mathbb{E}\Big[\inf_{0 \leq u \leq \tau_{\beta}}\Big(\frac{P_{\tau_{\beta}}}{P_{u}}\Big)\Big],
\eeq
\beq
\label{ab}
a:=\frac{(\beta-\delta-\lambda\alpha_s)}{\lambda(\alpha + \alpha_p -\alpha_s)}, \qquad \quad b:=\frac{c}{\lambda\kappa(\alpha + \alpha_p -\alpha_s)},
\eeq
\beq
\label{baselinear}
\ell^*_t = -\frac{1}{\gamma}\ln\Big( a + \frac{b}{P_t}\Big).
\eeq
Then $\ell^*_t \ $\footnote{The logarithm in \eqref{baselinear} is well defined since $a:=\frac{(\beta-\delta-\lambda\alpha_s)}{\lambda(\alpha + \alpha_p -\alpha_s)}>0$, by Assumption \ref{assr} and by the nonnegativity of $\alpha +\alpha_p-\alpha_s$, and $b>0$ since $\kappa \in [0,1]$.} solves the stochastic backward equation
\beq
\label{backlinear}
\mathbb{E}\bigg[\int_{\tau}^{\infty} \Gamma_y(t,\sup_{\tau \leq u \leq t}\ell^*_u) dt \Big|\mathcal{F}_{\tau}\bigg] = e^{-\beta \tau}P_{\tau}, \qquad \text{a.s.\ for any} \ \tau \geq 0,
\eeq
and hence the optimal procurement policy for problem \eqref{W} is
\beq
\label{optimalcontrol}
\overline{\nu}^*_t:= \sup_{0 \leq s < t}(\ell^*_s-y) \vee 0, \qquad \overline{\nu}^*_0=0.
\eeq

\end{proposition}

\begin{proof}
Substitute \eqref{Piderivative} in \eqref{backlinear}, make a change of variable in the integral, and use \eqref{Psupermartingale3} to obtain
\begin{eqnarray}
\label{back1}
\lefteqn{\mathbb{E}\bigg[\int_{\tau}^{\infty} \Gamma_y(t,\sup_{\tau \leq u \leq t}\ell^*_u) dt \Big|\mathcal{F}_{\tau}\bigg]  = e^{-\beta \tau} \mathbb{E}\bigg[\int_{0}^{\infty}e^{-\beta s}\Big(\lambda\alpha_s P_{s+\tau} - c\Big) ds\Big|\mathcal{F}_{\tau}\bigg]}   \\
& & \hspace{3.5cm}+\lambda(\alpha_p + \alpha -\alpha_s)e^{-\beta \tau}\mathbb{E}\bigg[\int_{0}^{\infty}e^{-\beta s}P_{s+\tau}\,e^{-\gamma \sup_{0 \leq u \leq s}(\ell^*_{u+\tau})}ds\Big|\mathcal{F}_{\tau}\bigg]  \nonumber\\
& & =e^{-\beta \tau}\Big(\frac{\lambda\alpha_s}{\beta-\delta} P_{\tau} - \frac{c}{\beta}\Big) + \lambda(\alpha + \alpha_p -\alpha_s)e^{-\beta \tau}\mathbb{E}\bigg[\int_{0}^{\infty}e^{-\beta s}P_{s+\tau}\,e^{\inf_{0 \leq u \leq s}(-\gamma\ell^*_{u+\tau})}ds\Big|\mathcal{F}_{\tau}\bigg]. \nonumber
\end{eqnarray}
Therefore we may rewrite \eqref{backlinear} in the equivalent form
\begin{eqnarray}
\label{equivalentback}
&& \hspace{-1.5cm} \mathbb{E}\bigg[\int_{0 }^{\infty}e^{-\beta s}P_{s+\tau}\,e^{\inf_{0 \leq u \leq s}(-\gamma\ell^*_{u+\tau})}ds\Big|\mathcal{F}_{\tau}\bigg] =
\frac{1}{\lambda(\alpha + \alpha_p -\alpha_s)}\Big[P_{\tau}\Big(1 - \frac{\lambda\alpha_s}{\beta-\delta}\Big) + \frac{c}{\beta}\Big].
\end{eqnarray}
Now we make a guess for the solution and we try with
\beq
\label{ansatz}
\ell^*_t = -\frac{1}{\gamma}\ln\Big(a + \frac{b}{P_t}\Big),
\eeq
for some positive $a$ and $b$. In fact the the left-hand side of \eqref{equivalentback} becomes
\begin{eqnarray}
\label{equivalentback2}
\lefteqn{\mathbb{E}\bigg[\int_{0}^{\infty}e^{-\beta s}P_{s+\tau}\,e^{\inf_{0 \leq u \leq s}\ln\big(a + \frac{b}{P_{u+\tau}}\big)}\,ds\Big|\mathcal{F}_{\tau}\bigg]} \nonumber \\
&=& \mathbb{E}\bigg[\int_{0}^{\infty}e^{-\beta s}P_{s+\tau}\,e^{\ln\big(\inf_{0 \leq u \leq s}(a + \frac{b}{P_{u+\tau}})\big)}\,ds\Big|\mathcal{F}_{\tau}\bigg]  \nonumber \\
&=&  \mathbb{E}\bigg[\int_{0}^{\infty}e^{-\beta s}P_{s+\tau}\,\inf_{0 \leq u \leq s}\Big(a + \frac{b}{P_{u+\tau}}\Big)\,ds\Big|\mathcal{F}_{\tau}\bigg]  \\
&=&  a\,\mathbb{E}\bigg[\int_{0}^{\infty}e^{-\beta s}P_{s+\tau}\,ds\Big|\mathcal{F}_{\tau}\bigg] + b\,\mathbb{E}\bigg[\int_{0}^{\infty}e^{-\beta s}\,\inf_{0 \leq u \leq s}\Big(\frac{P_{s+\tau}}{P_{u+\tau}}\Big)\,ds\Big|\mathcal{F}_{\tau}\bigg]  \nonumber \\
&=&  \frac{aP_{\tau}}{\beta-\delta} + \frac{b}{\beta}\mathbb{E}\bigg[\int_{0}^{\infty}\beta e^{-\beta s}\,\inf_{0 \leq u \leq s}\Big(\frac{P_{s}}{P_{u}}\Big)\,ds\bigg] \ \ =\ \ \frac{aP_{\tau}}{\beta-\delta} + \frac{b}{\beta}\kappa, \nonumber
\end{eqnarray}
therefore \eqref{equivalentback} holds (cf. \eqref{ab}).

Now optimality of \eqref{optimalcontrol} follows from Proposition \ref{baseandoptimalcontrol} if we show that $\overline{\nu}^*$ is admissible.
Clearly $\overline{\nu}^*$ is $\{\mathcal{F}_t\}$-adapted and left-continuous. Also
$\overline{\nu}^*_t \leq (-\frac{1}{\gamma}\ln(a) - y) \vee 0$ by monotonicity of $\ln$ and $b>0$. Therefore
$\mathbb{E}\big[\int_{0}^{\infty}e^{- \beta t} P_t \overline{\nu}^*_t dt \big] \leq \frac{(-\frac{1}{\gamma}\ln(a) - y) \vee 0}{\beta-\delta}<\infty$ by \eqref{Psupermartingale2}, and hence $\overline{\nu}^* \in \mathcal{S}$.
\end{proof}
\noindent The result of Proposition \ref{backlinearcostsprop} is remarkable in its own right. In fact, to the best of our knowledge, it is one of the rare examples of explicit solution to a backward stochastic equation like \eqref{backwardgeneric} involving the function $\Gamma_y(t,\cdot)$ (cf.\ \eqref{Piderivative}) not satisfying the classical Inada conditions.
\begin{remark}
Notice that if $\beta-\delta \geq \lambda(\alpha_p + \alpha)$, then $a \geq 1$ (cf.\ \eqref{ab}) and therefore $\ell^*_t <0$ for any $t \geq 0$ (cf.\ \eqref{baselinear}). It thus follows from \eqref{optimalcontrol} that $\overline{\nu}^*_t=0$ for all $t\geq 0$, and this is in line with Proposition \ref{noinvestprop}.
\end{remark}
In order to obtain explicitly the constant $\kappa$ of \eqref{kappa}, we now restrict the price $P$ dynamics to the class
of L\'{e}vy processes with no positive jumps.

\begin{proposition}
\label{propkappa}
Assume that $X$ is a L\'{e}vy process with no positive jumps if $\zeta>0$ and with no negative jumps if $\zeta<0$. Set
\beq
\label{Xtilde}
\widetilde{X}_u:=-\delta u + \pi(-\zeta)u + \zeta X_u,
\eeq
so that $P_t=e^{-\tilde{X}_t}$ (cf.\ \eqref{P}), and denote by $\widetilde{\pi}(\cdot)$ the Laplace exponent of $\widetilde{X}$. Then the constant $\kappa$ of \eqref{kappa} is
\beq
\label{kappa2}
\kappa =  \frac{\xi}{1+\xi},
\eeq
with $\xi$ uniquely determined by the equation $\widetilde{\pi}(\xi)=\beta$.
\end{proposition}
\begin{proof}
Starting from \eqref{kappa} we have
\begin{eqnarray*}
\kappa \hspace{-0.25cm} & = & \hspace{-0.25cm} \mathbb{E}\Big[e^{\inf_{0 \leq u \leq \tau_{\beta}}(\widetilde{X}_u - \widetilde{X}_{\tau_{\beta}})}\Big] = \mathbb{E}\Big[e^{-\widetilde{X}_{\tau_{\beta}} - \sup_{0 \leq u \leq \tau_{\beta}}(-\widetilde{X}_u)}\Big] \nonumber \\
\hspace{-0.25cm} & = & \hspace{-0.25cm} \mathbb{E}\Big[e^{\inf_{0 \leq u \leq \tau_{\beta}}(-\widetilde{X}_u)}\Big] = \mathbb{E}\Big[e^{-\sup_{0 \leq u \leq \tau_{\beta}}(\widetilde{X}_u)}\Big] = \frac{\xi}{1+\xi}, \nonumber
\end{eqnarray*}
since $-\widetilde{X}_{\tau_{\beta}} - \sup_{0 \leq u \leq \tau_{\beta}}(-\widetilde{X}_u) \sim \inf_{0 \leq u \leq \tau_{\beta}}(-\widetilde{X}_u)$ in distribution by the Duality Theorem for L\'{e}vy processes, and $\sup_{0 \leq u \leq \tau_{\beta}}(\widetilde{X}_u)$ is exponentially distributed with parameter $\xi$, with $\xi$ as defined above (cf. Bertoin, 1996, Chapter VII), due to the assumption of no positive jumps for $\widetilde{X}$.
\end{proof}

\begin{remark}
\label{remarkdoppijumpi}
Similar findings in the case of positive and negative jumps might be obtained by using the results in Kou and Wang (2003) on double exponential jump diffusion processes.
\end{remark}

\subsection{A Probabilistic Representation of the Value Function}
\label{valuefunctiolinearsection}

In the special setting of Section \ref{examples} we are able to provide a probabilistic representation of the value function \eqref{W}.
\begin{proposition}
\label{prop:valuefunction}
Let $\tau_{\beta-\delta}$ and $\tau_{\beta}$ be two independent, exponentially distributed random times with parameters ${\beta-\delta}$ and $\beta$, respectively. Then the value function \eqref{W} admits the representation
\begin{eqnarray}
\label{valuefunctionlinearcosts}
W(y) &\hspace{-0.25cm} = \hspace{-0.25cm}&  y + \frac{\lambda\alpha}{\gamma(\beta-\delta)} - \frac{\lambda (\alpha_p + \alpha - \alpha_s)}{\gamma(\beta-\delta)} \ \widetilde{\mathbb{E}}\Big[e^{-\gamma (y + \overline{\nu}^*_{\tau_{\beta-\delta}})}\Big]\nonumber \\
& & \hspace{0.2cm}  + \left(\frac{\lambda \alpha_s}{\beta - \delta}-1\right) \widetilde{\mathbb{E}}\Big[y + \overline{\nu}^*_{\tau_{\beta-\delta}} \Big]
-  \frac{c}{\beta}\mathbb{E}\Big[y + \overline{\nu}^*_{\tau_{\beta}}\Big],
\end{eqnarray}
in terms of the optimal control $\overline{\nu}^*$ of \eqref{optimalcontrol} and the expectations $\mathbb{E}[\cdot]$, and $\widetilde{\mathbb{E}}[\cdot]$  under the equivalent martingale measure of \eqref{Ptilde}.
\end{proposition}

\begin{proof}
Recall \eqref{Hlin}, \eqref{Pilin} and that $\beta=r+\lambda$. Then
\begin{eqnarray}
\label{optimalfunctional}
& & W(y)=\hat{\mathcal{J}}(y,\overline{\nu}^*)=\mathbb{E}\bigg[\int_0^{\infty}e^{-\beta t}\lambda P_t\Big[\alpha_s (y +\overline{\nu}^*_t) + \frac{\alpha}{\gamma} - \frac{(\alpha_p + \alpha - \alpha_s)}{\gamma}e^{-\gamma (y + \overline{\nu}^*_t)}\Big] dt \nonumber \\
& & \hspace{3.7cm}- c \int_0^{\infty}e^{-\beta t} (y + \overline{\nu}^*_t) dt - \int_0^{\infty} e^{-\beta t} P_t d\overline{\nu}^*_t\bigg].
\end{eqnarray}
By Lemma \ref{lemmaFubini} and by introducing $\tau_{\beta}$ and $\tau_{\beta-\delta}$, two independent, exponentially distributed random times with parameters $\beta$ and $\beta-\delta$, respectively,
arguments similar to those employed in the proof of Proposition \ref{backlinearcostsprop} allow us to rewrite the terms on the right hand side of \eqref{optimalfunctional} as follows,
$$\lambda\alpha_s\mathbb{E}\bigg[\int_0^{\infty}e^{-\beta t} P_t (y +\overline{\nu}^*_t) dt\bigg] = \frac{\lambda \alpha_s}{\beta - \delta} \widetilde{\mathbb{E}}\Big[y + \overline{\nu}^*_{\tau_{\beta-\delta}} \Big];$$
$$\lambda\frac{\alpha}{\gamma}\mathbb{E}\bigg[\int_0^{\infty}e^{-\beta t} P_t dt\bigg] = \frac{\lambda\alpha}{\gamma(\beta-\delta)};$$
$$\lambda\frac{(\alpha_p + \alpha - \alpha_s)}{\gamma}\mathbb{E}\bigg[\int_0^{\infty}e^{-\beta t} P_t e^{-\gamma (y + \overline{\nu}^*_t)}dt\bigg] = \frac{\lambda (\alpha_p + \alpha - \alpha_s)}{\gamma(\beta-\delta)}\widetilde{\mathbb{E}}\Big[e^{-\gamma (y + \overline{\nu}^*_{\tau_{\beta-\delta}})}\Big];$$
$$c \ \mathbb{E}\bigg[\int_0^{\infty}e^{-\beta t} (y + \overline{\nu}^*_t) dt\bigg] = \frac{c}{\beta}\mathbb{E}\Big[y + \overline{\nu}^*_{\tau_{\beta}}\Big];$$
$$\mathbb{E}\bigg[\int_0^{\infty} e^{-\beta t} P_t d\overline{\nu}^*_t\bigg] = (\beta-\delta)\mathbb{E}\bigg[\int_0^{\infty} e^{-\beta t} P_t \overline{\nu}^*_t dt\bigg] = \widetilde{\mathbb{E}}\Big[\overline{\nu}^*_{\tau_{\beta-\delta}}\Big].$$
Hence \eqref{valuefunctionlinearcosts} follows after some simple algebra.
\end{proof}

\begin{remark}
Notice that, if the process $\widetilde{X}$ of \eqref{Xtilde} has no negative jumps, the expectations in \eqref{valuefunctionlinearcosts} may be evaluated (at least numerically).
Indeed, recalling \eqref{optimalcontrol}, \eqref{baselinear} and $P_t=e^{-\tilde{X}_t}$ (cf.\ \eqref{P}), and letting $\tau_{\rho}$ be an independent, exponentially distributed random time with arbitrary (positive) parameter $\rho$, we have
\begin{eqnarray}
y + \overline{\nu}^*_{\tau_{\rho}} \hspace{-0.25cm} & =& \hspace{-0.25cm} y \vee \sup_{0 \leq u \leq \tau_{\rho}}\Big(-\frac{1}{\gamma}\ln\Big(a + \frac{b}{P_u}\Big)\Big) = y \vee \Big[-\frac{1}{\gamma}\ln\Big(\inf_{0 \leq u \leq \tau_{\rho}} \big(a + \frac{b}{P_u}\big)\Big)\Big] \\
\hspace{-0.25cm} & = & \hspace{-0.25cm} y \vee \Big[-\frac{1}{\gamma}\ln\Big( a + b \ e^{\inf_{0 \leq u \leq \tau_{\rho}} \widetilde{X}_u}\Big)\Big] = y \vee \Big[-\frac{1}{\gamma}\ln\Big( a + b \ e^{-\sup_{0 \leq u \leq \tau_{\rho}} (-\widetilde{X}_u)}\Big)\Big]
 \nonumber \\
\hspace{-0.25cm} & =: & \hspace{-0.25cm} y \vee \Big(-\frac{1}{\gamma}\ln\Big( a + b \ e^{-\widetilde{M}^{-}_{\tau_{\rho}}}\Big)\Big),
\end{eqnarray}
where $\widetilde{M}^{-}_t:=\sup_{0 \leq u \leq t} (-\widetilde{X}_u)$.
Now, since $-\widetilde{X}$ has no positive jumps, then $\widetilde{M}^{-}_{\tau_{\rho}}$ is exponentially distributed (cf. Chapter VII of Bertoin, 1996) with parameter $\xi$ uniquely determined by the equation $\widetilde{\pi}^{-}(\xi)=\rho$, where $\widetilde{\pi}^{-}(\xi)$ is the Laplace exponent of $-\widetilde{X}$ at $\xi$ (and clearly depends on the probability under which we are taking expectations, $\mathbb{Q}$ and $\widetilde{\mathbb{Q}}$).
\end{remark}

\subsection{Numerical Findings}

Here we provide computer drawings of the optimal inventory level for spot prices driven by a geometric Brownian motion or by an exponential jump-diffusion process. In order to do that, we find the explicit form of $\kappa$ (cf. \eqref{kappa2}), needed to obtain the base inventory $\ell^*$ (cf. \eqref{ab}, \eqref{baselinear}) and hence the optimal procurement strategy \eqref{optimalcontrol}. Then we compare the optimal expected return of our model with that of a modified version of the classical newsvendor model.\\
\indent In the computer drawings below we assume the following parameters: the firm's manager discount factor $r=0.05$, the initial inventory level $y=0$ (in line with the initial condition of the newsvendor model), the parameter of the exponentially distributed demand time $\lambda=5$, the unitary holding cost $c=1$, the premium factors in the revenue multiplier $G$ (see \eqref{gainT}) $\alpha=1.2, \alpha_p=0.8, \alpha_s=0.7$, the parameter of the exponentially distributed demand $\gamma=0.05$.

\subsubsection{Geometric Brownian Motion Price: the Optimal Strategy}
As in Guo et al.\ (2011), we take
\beq
\label{GBMdynamic}
\left\{
\begin{array}{ll}
dP_t=P_t(\mu dt+ \sigma dB_t)\\
P_0= 1,
\end{array}
\right.
\eeq
where $\mu \in \mathbb{R}$ and $\sigma>0$ are constants and $\{B_t, t\geq 0 \}$ is an exogenous one-dimensional standard Brownian motion.
Then
$$P_t=\exp\Big\{(\mu -\frac{\sigma^2}{2}) t + \sigma B_t\Big\},$$
which is of type \eqref{P} with $X := -B$, $\delta:=\mu$, $\zeta:=\sigma$ and $\pi(-\zeta)=\frac{\sigma^2}{2}$.
In this case $\widetilde{X}_u=(\frac{1}{2}\sigma^2 - \mu) u -\sigma B_u$, hence (cf. Proposition \ref{propkappa}, \eqref{kappa2})
\begin{eqnarray}
\label{kappaGBM}
\kappa=\mathbb{E}\Big[e^{-\sup_{0 \leq u \leq \tau_{\beta}}\widetilde{X}_u}\Big] = \frac{\theta_+}{1 + \theta_+},
\end{eqnarray}
where $\theta_+$ is the positive root of $\frac{\sigma^2}{2} x^2+(\frac{\sigma^2}{2}-\mu)x - \beta = 0$ (this is a well known result for a Brownian motion with drift).
\begin{figure}[h]
\begin{center}
\includegraphics[scale=0.6]{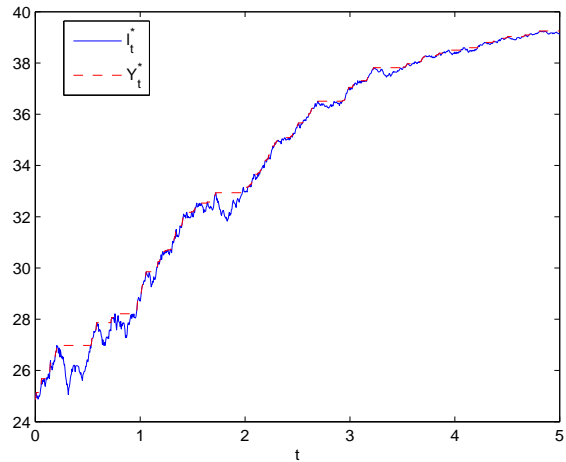}
\caption{Optimal procurement under geometric Brownian motion prices with $\mu=0.7$ and $\sigma=0.2$.}
\label{fig1}
\end{center}
\end{figure}
As shown in Figure \ref{fig1}, the commodity is purchased at times at which the current inventory becomes lower than the \textsl{base inventory} $\ell^*$ of \eqref{baselinear}.\bigskip

\subsubsection{Geometric Jump-Diffusion Price: the Optimal Strategy}
\label{subsubJump}
Since most spot prices exhibit significant skewness and kurtosis, they are not so well described by purely diffusive processes. Here we assume that the spot price evolves according to the \textsl{jump-diffusion} process
\beq
\label{Jumpdiffusiondynamics}
\left\{
\begin{array}{ll}
dP_t=P_{t^-}(\mu dt+ \sigma dB_t + dM_t)\\
P_{0^-}= 1,
\end{array}
\right.
\eeq
with $\mu$ and $\sigma$ real constants, $\{B_t, t\geq 0\}$ an exogenous one-dimensional standard Brownian motion, and
$$
dM_t:=\sum_{j=1}^{N_t}U_j,
$$
where $\{N_t, t\geq 0\}$ is a Poisson process with constant intensity $\psi \geq 0$, $\{U_j\}_{j \geq 0}$ is a sequence of i.i.d.\ random jumps with values in $[0,\infty)$ and such that $U_0=1$ a.s. Moreover, $B,N,U$ are independent.
Clearly this model includes deterministic growth models (for $\sigma = 0$) and pure jump models (for $\sigma = 0$ and $\psi > 0$).

The explicit solution of \eqref{Jumpdiffusiondynamics} is
$$\displaystyle P_t=\exp\Big\{(\mu-\frac{\sigma^2}{2})t+\sigma B_t\Big\} \prod_{j=1}^{N_t}(U_j + 1) = \exp\Big\{\displaystyle\big(\mu-\frac{\sigma^2}{2}\big)t+\sigma B_t+\sum_{j=1}^{N_t}Z_j\Big\}$$
where $Z_j:=\ln(U_j+1)$ are i.i.d. Hence $P_t=\exp\{-\widetilde{X}_t\}$ where $\widetilde{X}_t:=-(\mu-\frac{\sigma^2}{2}\big)t-\sigma B_t - \sum_{j=1}^{N_t}Z_j$ is
a L\'{e}vy process without positive jumps and with Laplace exponent
$$
\tilde{\pi}(u)=\frac{\sigma^2}{2} u^2+\big(\frac{\sigma^2}{2}-\mu\big)u+\psi\Big(\mathbb{E}(e^{-u Z_1})-1\Big),
$$
for any $u$ such that $E(e^{u \widetilde{X}_t})<\infty$. If we take the $Z_j$'s exponentially distributed with parameter $l>1$, then we get
$$
\tilde{\pi}(u)=\frac{\sigma^2}{2}u^2+u \Big(\frac{\sigma^2}{2}-\mu-\frac{\psi}{l+u}\Big),
$$
and by solving the equation $\widetilde{\pi}(\xi)=\beta$ for $\xi$, we may get $\kappa$ (cf. \eqref{kappa2}) and hence $\overline{\nu}^*$ (cf.\eqref{baselinear}).

\begin{figure}[h]
\begin{center}
\includegraphics[scale=0.6]{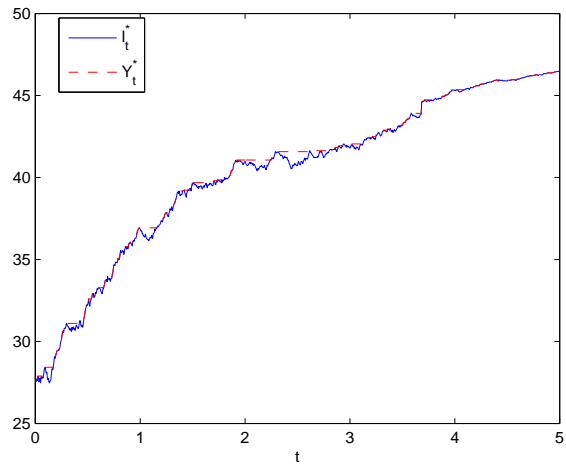}
\caption{Optimal inventory under geometric jump-diffusion prices with $\mu=0.7$, $\sigma=0.2$, $\psi=2$, $l=9$.}
\label{fig2}
\end{center}
\end{figure}

As it is shown in Figure \ref{fig2}, the optimal inventory is kept equal or higher than the \textsl{base inventory} \eqref{baselinear}, and purchases take place only when the current inventory becomes lower than the \textsl{base inventory}.

\subsubsection{Comparison with the Value Function of the Newsvendor Model}

The newsvendor model is a classical model in the literature on inventory management (see for example Porteus, 1990). It studies the problem of controlling the inventory of a single item with stochastic demand over a single period,
in the presence of overage and underage costs when the newsvendor orders too much or too little, respectively.
The newsvendor aims to choose the size of a single order that maximizes the expected profit.
As Guo et al.\ (2011) observe, the newsvendor model provides a solution for a firm's manager who is not interested in repeatedly buying in the spot market, but who purchases the commodity only once at time zero (the so called `newsvendor procurement strategy').
Such strategy belongs to our set of admissible strategies \eqref{admissibleset}. In the setting of this section (cf.\ beginning of Section \ref{examples}), the firm's total expected discounted return associated to the purchase of $y$ units of commodity at time zero is given (cf. \eqref{return}) by
\beq
\label{newsfunc}
L(y):=\mathbb{E}\bigg[e^{-r \Theta}P_{\Theta} \ G(y,D)-\big(1+\frac{c}{r}(1-e^{-r\Theta})\big)y\bigg],
\eeq
and the newsvendor value function is
$$\max_{y\geq0} L(y).$$
Recalling that $F_D(z)=1 - e^{-\gamma z}$, $\gamma>0$, simple calculations give that the optimal newsvendor procurement strategy is
\[
y^*= \max \Big{\{} 0, F^{-1}_D(\eta) \Big{\}},
\]
where
\[
\eta=\frac{(\alpha+\alpha_p)\mathbb{E}\big[e^{-r\Theta}P_\Theta \big]-\big(1+\frac{c}{r}(1-\mathbb{E}\big[e^{-r\Theta}\big])\big)}{(\alpha+\alpha_p-\alpha_s)\mathbb{E}\big[e^{-r\Theta}P_\Theta \big]}.
\]
Assuming the geometric Brownian motion of \eqref{GBMdynamic} with $\mu=0.7$ for the spot price, we use the probabilistic representation \eqref{valuefunctionlinearcosts}  to plot the difference between the value function $V(0)$ (cf. \eqref{equivalentvaluefunction0} with $y=0$) and the newsvendor value function $L(y^*)$ as the volatility $\sigma$ varies from $0.05$ to $50$ (see Figure \ref{fig3}). Notice that when $\sigma\rightarrow0$ our value function $V(0)$ tends to the newsvendor's one $L(y^*)$. On the other hand, when the volatility increases, the optimal revenue starts rapidly becoming larger than the newsvendor revenue. That means, the more the market gets riskier the more profitable the optimal dynamic procurement strategy becomes.
\begin{figure}
\begin{center}
\includegraphics[scale=0.6]{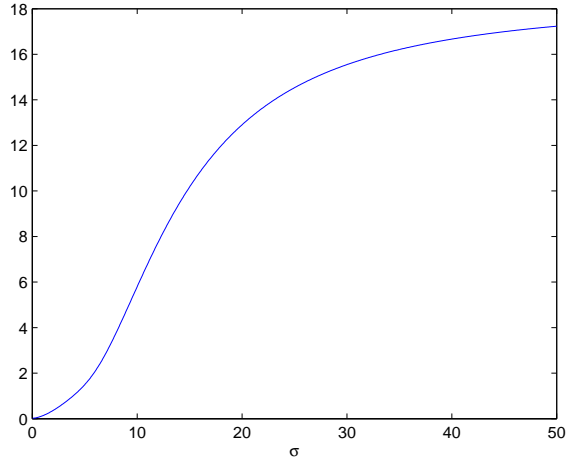}
\caption{Difference between $V(0)$ and $L(y^*)$  when $\sigma$ varies from $0.05$ to $50$.}
\label{fig3}
\end{center}
\end{figure}

\section{Conclusions and Future Research}
\label{conclusions}
We have studied the optimal procurement problem of a firm aiming at meeting a random demand at a random time $\Theta$. The firm buys inventory on a spot market and maximizes expected profits under exponential L\'{e}vy spot prices and general convex running holding costs. The problem is modeled as a monotone stochastic control problem in which the cumulative inventory up to time $t$ is the control.

We prove existence of an optimal procurement policy and we characterize it through stochastic first order conditions. As expected in singular stochastic control, it is optimal for the firm to invest in the spot market just enough to keep, at any time, the inventory level above a certain lower bound. In fact, in the present case the lower bound is the \textsl{base inventory} process, which is random and time dependent. \\
\indent In the case of linear holding costs and an exponentially distributed random demand, we obtain a closed form solution for the optimal procurement policy, which we use to provide computer drawings for two examples of exponential L\'{e}vy prices. In particular, for geometric Brownian motion prices Figure \ref{fig3} shows that, as the volatility increases, the optimal procurement strategy becomes more and more profitable than the static newsvendor's one.

The results of this paper are based on a quite recent and powerful stochastic first order condition approach, known in the mathematical finance literature (see, e.g., Bank \& Riedel, 2001; Riedel \& Su, 2011), but still rarely employed in the operational research context. In this sense our paper contributes to the literature on continuous time inventory management policies. The new method of solution provides new insights into the optimal procurement under L\'evy prices and general convex holding costs.

There are many directions in which it would be interesting to extend the present study as, for example, relaxing the independence assumption between prices and demand and introducing some form of correlation between them, or assuming diffusive dynamics also for the demand $D$ (that would give raise to a daunting three-dimensional singular stochastic control problem in $(P_t,D_t,Y_t)$).
It would be also interesting to allow the firm to buy and sell in the spot market, i.e. to allow bounded variation stochastic controls.
Finally, one could introduce fixed inventory ordering costs. That would naturally lead to a challenging impulse control problem in which a strategy of S--s type (see Porteus, 1990) is expected to be optimal.
More drastically, one could try to drop the Markovian setting. In such case, the stochastic first order conditions might still hold, but it would be almost impossible to find explicit solutions.

\vspace{1cm}
\indent \textbf{Acknowledgments.} The authors wish to thank two anonymous referees for their helpful comments on an earlier version of this paper.

\end{document}